\documentclass[12pt]{amsproc}
\usepackage{amsmath,amssymb,amsthm,amsthm,mathtools,enumitem}
\usepackage[top=30mm, bottom=35mm, left=30mm, right=30mm]{geometry}
\DeclareFontEncoding{OT2}{}{}
\DeclareFontSubstitution{OT2}{cmr}{m}{n}
\DeclareFontFamily{OT2}{cmr}{\hyphenchar\font45 }
\DeclareFontShape{OT2}{cmr}{m}{n}{
<5><6><7><8><9>gen*wncyr
<10><10.95><12><14.4><17.28><20.74><24.88>wncyr10}{}
\DeclareFontShape{OT2}{cmr}{b}{n}{
<5><6><7><8><9>gen*wncyb
<10><10.95><12><14.4><17.28><20.74><24.88>wncyb10}{}
\DeclareMathAlphabet{\mathcyr}{OT2}{cmr}{m}{n}
\DeclareMathAlphabet{\mathcyb}{OT2}{cmr}{b}{n}
\SetMathAlphabet{\mathcyr}{bold}{OT2}{cmr}{b}{n}
\newcommand{\genstirlingI}[3]{%
  \genfrac{[}{]}{0pt}{#1}{#2}{#3}%
}

\newcommand{\stirlingI}[2]{\genstirlingI{}{#1}{#2}}

\usepackage{hyperref}
\usepackage{xcolor}
\usepackage[capitalize,nameinlink,noabbrev,nosort]{cleveref}
\hypersetup{
	colorlinks=true,       % false: boxed links; true: colored links
	linkcolor=brown,          % color of internal links
	citecolor=brown,        % color of links to bibliography
	filecolor=brown,      % color of file links
	urlcolor=brown,           % color of external links
}
\newtheorem{theoremcounter}{Theorem Counter}[section]
\theoremstyle{plain}
\newtheorem{theorem}[theoremcounter]{Theorem}
\newtheorem{proposition}[theoremcounter]{Proposition}
\newtheorem{lemma}[theoremcounter]{Lemma}
\theoremstyle{remark}
\newtheorem{definition}[theoremcounter]{Definition}
\newtheorem{remark}[theoremcounter]{Remark}
\newcommand{\bn}{\boldsymbol{n}}
\newcommand{\cA}{\mathcal{A}}
\newcommand{\BB}{\mathbb{B}}
\newcommand{\dif}{\mathrm{d}}
\newcommand{\emp}{\varnothing}
\newcommand{\fH}{\mathfrak{H}}
\newcommand{\bh}{\boldsymbol{h}}
\newcommand{\bk}{\boldsymbol{k}}
\newcommand{\bl}{\boldsymbol{l}}
\newcommand{\bm}{\boldsymbol{m}}
\newcommand{\bx}{\boldsymbol{x}}
\newcommand{\bxi}{\boldsymbol{\xi}}
\newcommand{\bz}{\boldsymbol{z}}
\newcommand{\Li}{\mathrm{Li}}
\newcommand{\ii}{\mathrm{I}}
\newcommand{\PP}{\mathcal{P}}
\newcommand{\QQ}{\mathbb{Q}}
\newcommand{\rL}{\mathrm{L}}

\newcommand{\ZZ}{\mathbb{Z}}
\newcommand{\CC}{\mathbb{C}}
\newcommand{\dep}{\mathrm{dep}}
\newcommand{\wt}{\mathrm{wt}}
\newcommand{\sh}{\mathbin{\mathcyr{sh}}}
\numberwithin{equation}{section}

\allowdisplaybreaks
\address{Institute for Advanced Research, Nagoya University, Furo-cho, Chikusa-ku, Nagoya, 464-8602, Japan}
\email{minoru.hirose@math.nagoya-u.ac.jp}
\address{Faculty of Mathematics, Kyushu University, Motooka 744, Nishi-ku, Fukuoka, 819-0395, Japan}
\email{matsusaka@math.kyushu-u.ac.jp}
\address{Department of Mathematical Sciences, Aoyama Gakuin University, 5-10-1 Fuchinobe, Chuo-ku, Sagamihara, Kanagawa, 252-5258, Japan}
\email{seki@math.aoyama.ac.jp}
\thanks{This research was supported by JSPS KAKENHI Grant Numbers JP18K13392, JP22K03244 (Hirose), JP20K14292, JP21K18141, JP24K16901 (Matsusaka), and JP21K13762 (Seki).}
\title[]{A discretization of the iterated integral expression of the multiple polylogarithm}
\author[]{Minoru Hirose, Toshiki Matsusaka, and Shin-ichiro Seki}
\date{}
\begin{document}
% --------------------------------------------------------------------------
\maketitle
\begin{abstract}
Recently, Maesaka, Watanabe, and the third author discovered a phenomenon where the iterated integral expressions of multiple zeta values become discretized.
In this paper, we extend their result to the case of multiple polylogarithms and provide two proofs.
The first proof uses the method of connected sums, while the second employs induction based on the difference equations that discrete multiple polylogarithms satisfy.
We also investigate several applications of our main result.
\end{abstract}
% --------------------------------------------------------------------------
\section{Introduction}\label{sec:intro}
% --------------------------------------------------------------------------
The \emph{multiple polylogarithm} (MPL) $\Li_{\bk}^{\sh}(\bz)$ is defined as
\[
\Li_{\bk}^{\sh}(\bz)\coloneqq\sum_{0<n_1<\cdots<n_r}\frac{1}{n_1^{k_1}\cdots n_r^{k_r}}\prod_{i=1}^{r}\left(\frac{z_{i+1}}{z_i}\right)^{n_i},
\]
where a tuple $\bk=(k_1,\dots,k_r)$, consisting of positive integers, is called an index, $\bz=(z_1,\dots,z_r)$ is a tuple of complex numbers, and $z_{r+1}\coloneqq 1$.
For convergence, we assume that each $|z_i|\geq 1$ and $(k_r,z_r)\neq (1,1)$.
For the index $\bk=(k_1,\dots,k_r)$, we define its depth as $\dep(\bk)\coloneqq r$ and its weight as $\wt(\bk)\coloneqq k_1+\cdots+k_r$.
When all $z_i=1$, we denote $\Li_{\bk}(\{1\}^r)$ by $\zeta(\bk)$ and call it the \emph{multiple zeta value} (MZV).
Here, $\{1\}^r$ is a shorthand notation indicating that the number $1$ is repeated $r$ times.
An index $\bk$ that satisfies the convergence condition $k_r\geq 2$ is called an admissible index.
In the notation of iterated integrals
\[
\int_0^1\frac{\dif t}{t-a_1}\circ\cdots\circ\frac{\dif t}{t-a_k}\coloneqq\int_{0<t_1<\cdots<t_k<1}\frac{\dif t_1}{t_1-a_1}\cdots\frac{\dif t_k}{t_k-a_k},
\]
for $\bk$ and $\bz$, we define $\ii_{\bk}(\bz)$ as
\[
\ii_{\bk}(\bz)\coloneqq\int_0^1\underbrace{\frac{\dif t}{t-z_1}\circ\frac{\dif t}{t}\circ\cdots\circ\frac{\dif t}{t}}_{k_1\text{ times}}\circ\cdots\circ\underbrace{\frac{\dif t}{t-z_r}\circ\frac{\dif t}{t}\circ\cdots\circ\frac{\dif t}{t}}_{k_r\text{ times}}.
\]
Then, the following iterated integral expression of the MPL holds (\cite[Theorem~0.16]{Goncharov1995}, \cite[Theorem~2.2]{Goncharov-pre}):
\begin{equation}\label{eq:IIEforMPL}
\Li_{\bk}^{\sh}(\bz)=(-1)^{\dep(\bk)}\ii_{\bk}(\bz).
\end{equation}
In particular, the iterated integral expression of the MZV,
\begin{equation}\label{eq:IIEforMZV}
\zeta(\bk)=(-1)^{\dep(\bk)}\ii_{\bk}(\{1\}^{\dep(\bk)}),
\end{equation}
is obtained.
Recently, a discretization phenomenon of this expression was discovered by Maesaka, Watanabe, and the third author.
Let $N$ be a positive integer.
\begin{theorem}[Maesaka--Seki--Watanabe~\cite{MaesakaSekiWatanabe-pre}]\label{thm:MSW}
For any index $\bk=(k_1,\dots,k_r)$ and any $N$, we have
\[
\sum_{0<n_1<\cdots<n_r<N}\frac{1}{n_1^{k_1}\cdots n_r^{k_r}}=(-1)^r\sum_{\substack{0<n_{j,1}\leq\cdots\leq n_{j,k_j} < N \ (1\leq j\leq r)\\ n_{j,k_j}<n_{j+1,1} \ (1\leq j<r)}}\prod_{j=1}^r\frac{1}{(n_{j,1}-N)n_{j,2}\cdots n_{j,k_j}}.
\]
\end{theorem}
The sum on the right-hand side is a Riemann sum of the iterated integral, and taking the limit of both sides as $N\to\infty$, when $\bk$ is admissible, yields \eqref{eq:IIEforMZV}.
\cref{thm:MSW} asserts that the integral expression remains valid in a discrete form even before taking the limit.

Building upon this work, it is desired to investigate how widely such a discretization phenomenon of ``series = integral'' type expressions holds. Yamamoto \cite{Yamamoto-pre} has already shown that the integral expressions associated with the 2-posets of the multiple zeta-star values given in \cite{Yamamoto2017}, as well as those extended to the Schur multiple zeta values of diagonally constant indices provided by the first author, Murahara, and Onozuka in \cite{HiroseMuraharaOnozuka-pre}, are similarly discretized.

In this paper, we provide a discretization of the iterated integral expression of the MPL \eqref{eq:IIEforMPL}.
In the following, when dealing with finite sums, we describe them using indeterminates $x_i$ instead of complex numbers $z_i$.

As mentioned in \cite{MaesakaSekiWatanabe-pre}, for example, we can consider the elementary equation
\begin{equation}\label{eq:MSWexample}
\sum_{n=1}^{2N-1}\frac{(-1)^{n-1}}{n}=\sum_{n=0}^{N-1}\frac{1}{n+N}
\end{equation}
as a discretization of 
\[
\log 2=\sum_{n=1}^{\infty}\frac{(-1)^{n-1}}{n}=\int_0^1\frac{\mathrm{d}t}{t+1}.
\]
It seems to be natural to introduce the following approximate finite sums for $\Li_{\bk}^{\sh}(\bz)$ and $\ii_{\bk}(\bz)$, respectively: for $\bk=(k_1,\dots, k_r)$ and $\bx=(x_1,\dots, x_r)$ ($x_{r+1}\coloneqq 1$),
\begin{align*}
&\Li_{\bk}^{\sh,< N}(\bx)\coloneqq\sum_{0<n_1<\cdots<n_r< N}\frac{1}{n_1^{k_1}\cdots n_r^{k_r}}\prod_{i=1}^{r}\left(\frac{x_{i+1}}{x_i}\right)^{n_i},\\
&\ii_{\bk}^{(N)}(\bx)\coloneqq\sum_{\substack{0< n_{j,1}\leq\cdots\leq n_{j,k_j}< N \ (1\leq j\leq r)\\ n_{j,k_j}<n_{j+1,1} \ (1\leq j<r)}}\prod_{j=1}^r\frac{1}{(n_{j,1}-Nx_j)n_{j,2}\cdots n_{j,k_j}}.
\end{align*}
However, comparing the values of $\Li_{\bk}^{\sh,<N}(\bx)$ and $(-1)^{\dep(\bk)}\ii_{\bk}^{(N)}(\bx)$, while adjusting the sum ranges with \eqref{eq:MSWexample} in mind, generally does not reveal a simple relationship.
On the other hand, by modifying the definition of $\Li_{\bk}^{\sh,<N}(\bx)$, for example, the following equation can be found: for a positive integer $k$,
\[
\sum_{n=1}^{N-1}\frac{(-1)^{n-1}}{n^k}\frac{\binom{N-1}{n}}{\binom{N+n}{n}}=\sum_{0 < n_1\leq\cdots\leq n_k < N}\frac{1}{(n_1+N)n_2\cdots n_k}
\]
or equivalently
\[
\sum_{n=1}^{N-1}\frac{1}{n^k}\frac{\binom{N-1}{n}}{\binom{-N-1}{n}}=-\sum_{0< n_1\leq\cdots\leq n_k< N}\frac{1}{(n_1+N)n_2\cdots n_k}.
\]
By taking the limit of this equation as $N\to\infty$, the expression $\Li_{k}^{\sh}(-1)=-\ii_k(-1)$ can be obtained (cf.~\cite[Lemma~4.2]{LinebargerZhao2015}).
In fact, this can be fully extended, and the following is our main result.
The generalized binomial coefficient $\binom{x}{n}$ as a polynomial is defined in the usual way for non-negative integer $n$: $\binom{x}{n}\coloneqq\frac{x(x-1)\cdots(x-n+1)}{n(n-1)\cdots 1}$ ($n\geq 1$) and $\binom{x}{0}\coloneqq1$.
\begin{theorem}[Discretization of the iterated integral expression of the MPL]\label{thm:main}
Let $\bk=(k_1,\dots, k_r)$ be an index and $\bx=(x_1,\dots, x_r)$ a tuple of indeterminates with $x_{r+1} = 1$.
Then, for any positive integer $N$, we have
\begin{align*}
&\sum_{0<n_1<\cdots<n_r<N}\frac{1}{n_1^{k_1}\cdots n_r^{k_r}}\prod_{i=1}^{r}\frac{\binom{Nx_{i+1}-1}{n_i}}{\binom{Nx_i-1}{n_i}}\\
&=(-1)^r\sum_{\substack{0< n_{j,1}\leq\cdots\leq n_{j,k_j}< N \ (1\leq j\leq r)\\ n_{j,k_j}<n_{j+1,1} \ (1\leq j<r)}}\prod_{j=1}^r\frac{1}{(n_{j,1}-Nx_j)n_{j,2}\cdots n_{j,k_j}}.
\end{align*}
\end{theorem}
By abbreviating the left-hand side as $\widetilde{\Li}_{\bk}^{\sh,(N)}(\bx)$, it can be expressed as
\begin{equation}\label{eq:main}
\widetilde{\Li}_{\bk}^{\sh,(N)}(\bx)=(-1)^{\dep(\bk)}\ii_{\bk}^{(N)}(\bx).
\end{equation}
For a fixed $n_i$'s, it is clear that
\begin{equation}\label{eq:term-wise-limit}
\lim_{N\to\infty}\prod_{i=1}^{r}\frac{\binom{Nz_{i+1}-1}{n_i}}{\binom{Nz_i-1}{n_i}}=\prod_{i=1}^{r}\left(\frac{z_{i+1}}{z_i}\right)^{n_i}.
\end{equation}
Hence, when $\bk$ and $\bz$ satisfy the convergence condition and after substituting $x_i=z_i$, we expect that the limit of \eqref{eq:main} as $N \to \infty$ yields \eqref{eq:IIEforMPL}.
Strictly speaking, \eqref{eq:term-wise-limit} alone is insufficient as justification; however this is justified by \cref{prop:asymptotic_for_modifiedMPL}.

By substituting $1$ for all $x_i$ in \eqref{eq:main}, it becomes exactly \cref{thm:MSW}.
Similar to the proof of \cref{thm:MSW} in \cite{MaesakaSekiWatanabe-pre}, \cref{thm:main} can be proved using the method of connected sums.

Let $\stirlingI{n}{j}$ denote the (unsigned) Stirling number of the first kind defined by $x(x+1)\cdots (x+n-1)=\sum_{j=0}^n\stirlingI{n}{j}x^j$.
From our main result, as a discretization of 
\[
\frac{\pi}{4}=\sum_{n=0}^{\infty}\frac{(-1)^n}{2n+1}=\int_0^1\frac{\dif t}{t^2+1},
\]
we have
\begin{equation}\label{eq:pi/4}
\sum_{n=1}^{N-1}\frac{a_n^{(N)}}{n}=\sum_{n=1}^{N-1}\frac{N}{n^2+N^2},
\end{equation}
where
\[
a_n^{(N)}\coloneqq\left(\prod_{i=1}^n\frac{N-i}{i^2+N^2}\right)\sum_{0\leq j<n/2}(-1)^{n+j+1}\stirlingI{n+1}{2j+2}N^{2j+1}
\]
and
\[
\lim_{N\to\infty}a_n^{(N)}=\begin{cases}(-1)^{\frac{n-1}{2}} & \text{if } n \text{ is odd},\\ 0 & \text{if } n \text{ is even.}\end{cases}
\]
See \cref{sec:Misc} for the proof.
In this way, it has observed that in the given ``series = integral'' expression, the expression can sometimes be discretized not merely by truncating the range of summation for the series, but also by replacing the summand with one that asymptotically approaches the original summand as $N\to\infty$.

It should be noted that, unlike MZVs, MPLs can be considered as functions with variables, which makes them differentiable.
In fact, it is known that MPLs satisfy the following differential formula due to Goncharov.
\begin{theorem}[{Goncharov~\cite[Theorem~2.1]{Goncharov-pre}}]\label{thm:Goncharov_diff}
Let $\bk=(k_1,\dots, k_r)$ be an index and $\bz=(z_1,\dots,z_r)$ complex variables with $|z_i|>1$.
We use the following abbreviated notation for $1\leq i\leq r$$:$ 
\begin{align*}
\bk^{\wedge}_{i}&\coloneqq(k_1,\dots,k_{i-1},k_{i+1},\dots,k_r),\\ 
\bk^{\downarrow}_{i}&\coloneqq(k_1,\dots,k_{i-1},k_i-1,k_{i+1},\dots,k_r) \quad (k_i>1), \\
\bz^{\wedge}_{i}&\coloneqq(z_1,\dots,z_{i-1},z_{i+1},\dots,z_r).
\end{align*}
Then the following holds for each case.
\begin{enumerate}[leftmargin=6mm,font=\normalfont] 
\item\label{it:Goncharov_dif_1} $i=1$.
    \begin{enumerate}[leftmargin=6mm,font=\normalfont] 
    \item\label{it:Goncharov_dif_1a} $k_1>1$.
    \[
    \frac{\partial}{\partial z_1}\ii_{\bk}(\bz)=-\frac{1}{z_1}\ii_{\bk^{\downarrow}_1}(\bz).
    \]
    \item\label{it:Goncharov_dif_1b} $r=1$ and $k_1=1$.
    \[
    \frac{\partial}{\partial z_1}\ii_{\bk}(\bz)=\frac{1}{z_1-1}-\frac{1}{z_1}.
    \]
    \item\label{it:Goncharov_dif_1c} $r>1$ and $k_1=1$.
    \[
    \frac{\partial}{\partial z_1}\ii_{\bk}(\bz)=-\frac{1}{z_1}\ii_{\bk^{\wedge}_1}(\bz^{\wedge}_1)+\frac{1}{z_1-z_2}\left(\ii_{\bk^{\wedge}_1}(\bz^{\wedge}_1)-\ii_{\bk^{\wedge}_1}(\bz^{\wedge}_2)\right).
    \]
    \end{enumerate}
\item\label{it:Goncharov_dif_2} $r>1$ and $1<i<r$.
    \begin{enumerate}[leftmargin=6mm,font=\normalfont]
    \item\label{it:Goncharov_dif_2a} $k_{i-1}>1$, $k_i>1$.
    \[
    \frac{\partial}{\partial z_i}\ii_{\bk}(\bz)=\frac{1}{z_i}\left(\ii_{\bk^{\downarrow}_{i-1}}(\bz)-\ii_{\bk^{\downarrow}_i}(\bz)\right).
    \]
    \item\label{it:Goncharov_dif_2b} $k_{i-1}>1$, $k_i=1$.
    \[
    \frac{\partial}{\partial z_i}\ii_{\bk}(\bz)=\frac{1}{z_i}\left(\ii_{\bk^{\downarrow}_{i-1}}(\bz)-\ii_{\bk^{\wedge}_{i}}(\bz^{\wedge}_i)\right)+\frac{1}{z_i-z_{i+1}}\left(\ii_{\bk^{\wedge}_{i}}(\bz^{\wedge}_{i})-\ii_{\bk^{\wedge}_{i}}(\bz^{\wedge}_{i+1})\right).
    \]
    \item\label{it:Goncharov_dif_2c} $k_{i-1}=1$, $k_i>1$.
    \[
    \frac{\partial}{\partial z_i}\ii_{\bk}(\bz)=\frac{1}{z_i-z_{i-1}}\left(\ii_{\bk^{\wedge}_{i-1}}(\bz^{\wedge}_{i-1})-\ii_{\bk^{\wedge}_{i-1}}(\bz^{\wedge}_{i})\right)+\frac{1}{z_i}\left(\ii_{\bk^{\wedge}_{i-1}}(\bz^{\wedge}_{i})-\ii_{\bk^{\downarrow}_{i}}(\bz)\right).
    \]
    \item\label{it:Goncharov_dif_2d} $k_{i-1}=k_i=1$.
    \[
    \frac{\partial}{\partial z_i}\ii_{\bk}(\bz)=\frac{1}{z_i-z_{i-1}}\left(\ii_{\bk^{\wedge}_{i-1}}(\bz^{\wedge}_{i-1})-\ii_{\bk^{\wedge}_{i}}(\bz^{\wedge}_{i})\right)+\frac{1}{z_i-z_{i+1}}\left(\ii_{\bk^{\wedge}_{i}}(\bz^{\wedge}_{i})-\ii_{\bk^{\wedge}_{i}}(\bz^{\wedge}_{i+1})\right).
    \]
    \end{enumerate}
\item\label{it:Goncharov_dif_3} $i=r>1$.
By interpreting as $z_{r+1}=1$ and $\ii_{\bk^{\wedge}_r}(\bz^{\wedge}_{r+1})=0$, formulas identical to \eqref{it:Goncharov_dif_2} hold.
\end{enumerate}
\end{theorem}
We discretize Goncharov's differential formula and determine the difference equations satisfied by discrete multiple polylogarithms.
Here, we use the following notation:
\begin{align*}
\frac{\Delta^{(N)}f(\bx)}{\Delta^{(N)}x_i}&\coloneqq\frac{f(x_1,\dots,x_{i-1},x_i+N^{-1},x_{i+1},\dots,x_r)-f(\bx)}{N^{-1}},\\
f(\bx)\big|_{x_i+N^{-1}}&\coloneqq f(x_1,\dots,x_{i-1},x_i+N^{-1},x_{i+1},\dots,x_r)=f(\bx)+\frac{1}{N}\frac{\Delta^{(N)}f(\bx)}{\Delta^{(N)}x_i}.
\end{align*}
\begin{theorem}[Discretization of Goncharov's differential formula]\label{thm:diff}
Let $\bk=(k_1,\dots, k_r)$ be an index and $\bx=(x_1,\dots,x_r)$ a tuple of indeterminates.
The symbols $\bk^{\wedge}_{i}$, $\bk^{\downarrow}_i$, and $\bx^{\wedge}_i$ are used in a manner similar to their use in \cref{thm:Goncharov_diff}.
Then the following holds for each case.
\begin{enumerate}[leftmargin=6mm,font=\normalfont]
\item\label{it:discrete_dif_1} $i=1$.
    \begin{enumerate}[leftmargin=6mm,font=\normalfont]
    \item\label{it:discrete_dif_1a} $k_1>1$.
    \[
    \frac{\Delta^{(N)}\ii_{\bk}^{(N)}(\bx)}{\Delta^{(N)}x_1}=-\frac{1}{x_1}\ii_{\bk^{\downarrow}_1}^{(N)}(\bx).
    \]
    \item\label{it:discrete_dif_1b} $r=1$ and $k_1=1$.
    \[
    \frac{\Delta^{(N)}\ii_{\bk}^{(N)}(\bx)}{\Delta^{(N)}x_1}=\frac{1}{x_1+N^{-1}-1}-\frac{1}{x_1}.
    \]
    \item\label{it:discrete_dif_1c} $r>1$ and $k_1=1$.
    \[
    \frac{\Delta^{(N)}\ii_{\bk}^{(N)}(\bx)}{\Delta^{(N)}x_1} = -\frac{1}{x_1}\ii_{\bk^{\wedge}_1}^{(N)}(\bx^{\wedge}_1)+\frac{1}{x_1+N^{-1}-x_2}\left(\ii_{\bk^{\wedge}_1}^{(N)}(\bx^{\wedge}_1)-\ii_{\bk^{\wedge}_1}^{(N)}(\bx^{\wedge}_2)\bigg|_{x_1+N^{-1}}\right).
    \]
    \end{enumerate}
\item\label{it:discrete_dif_2} $r>1$ and $1<i<r$.
    \begin{enumerate}[leftmargin=6mm,font=\normalfont]
    \item\label{it:discrete_dif_2a} $k_{i-1}>1$, $k_i>1$.
    \[
    \frac{\Delta^{(N)}\ii_{\bk}^{(N)}(\bx)}{\Delta^{(N)}x_i} = \frac{1}{x_i}\left(\ii^{(N)}_{\bk^{\downarrow}_{i-1}}(\bx)\bigg|_{x_i+N^{-1}}-\ii^{(N)}_{\bk^{\downarrow}_i}(\bx)\right).
    \]
    \item\label{it:discrete_dif_2b} $k_{i-1}>1$, $k_i=1$.
    \begin{align*}
    \frac{\Delta^{(N)} \ii_{\bk}^{(N)}(\bx)}{\Delta^{(N)} x_i} &= \frac{1}{x_i}\left(\ii_{\bk_{i-1}^\downarrow}^{(N)}(\bx)\bigg|_{x_i+N^{-1}}-\ii_{\bk_i^\wedge}^{(N)}(\bx_i^\wedge)\right)\\
    &\quad+\frac{1}{x_i+N^{-1}-x_{i+1}}\left(\ii_{\bk_i^\wedge}^{(N)}(\bx_i^\wedge) - \ii_{\bk_i^\wedge}^{(N)}(\bx_{i+1}^\wedge)\bigg|_{x_i+N^{-1}}\right)\\
    &\quad+\frac{1}{N}\frac{1}{x_i(x_i+N^{-1}-x_{i+1})}\left(\ii^{(N)}_{(\bk_{i-1}^\downarrow)_i^\wedge}(\bx_{i+1}^\wedge)\bigg|_{x_i+N^{-1}} - \ii^{(N)}_{(\bk_{i-1}^\downarrow)_i^\wedge}(\bx_i^\wedge)\right).
    \end{align*}
    \item\label{it:discrete_dif_2c} $k_{i-1}=1$, $k_i>1$.
    \begin{align*}
    \frac{\Delta^{(N)}\ii_{\bk}^{(N)}(\bx)}{\Delta^{(N)}x_i}&=\frac{1}{x_i-x_{i-1}} \left(\ii_{\bk_{i-1}^\wedge}^{(N)}(\bx_{i-1}^\wedge) - \ii_{\bk_{i-1}^\wedge}^{(N)}(\bx_i^\wedge) \right) + \frac{1}{x_i} \left(\ii_{\bk_{i-1}^\wedge}^{(N)}(\bx_i^\wedge) - \ii_{\bk_i^\downarrow}^{(N)}(\bx)\right)\\
    &\quad+\frac{1}{N}\frac{1}{x_i(x_i-x_{i-1})}\left(\ii_{(\bk^{\downarrow}_i)^{\wedge}_{i-1}}^{(N)}(\bx^{\wedge}_i)-\ii_{(\bk^{\downarrow}_i)^{\wedge}_{i-1}}^{(N)}(\bx^{\wedge}_{i-1})\right).
    \end{align*}
    \item\label{it:discrete_dif_2d} $k_{i-1}=k_i=1$.
    \begin{align*}
    \frac{\Delta^{(N)} \ii_{\bk}^{(N)}(\bx)}{\Delta^{(N)} x_i} &= \frac{1}{x_i - x_{i-1}} \left(\ii_{\bk_{i-1}^\wedge}^{(N)}(\bx_{i-1}^\wedge) - \ii_{\bk_i^\wedge}^{(N)} (\bx_i^\wedge)\right) \\
    &\quad+ \frac{1}{x_i+N^{-1}-x_{i+1}} \left(\ii_{\bk_i^\wedge}^{(N)}(\bx_i^\wedge) - \ii_{\bk_i^\wedge}^{(N)} (\bx_{i+1}^\wedge)\bigg|_{x_i+N^{-1}} \right).
\end{align*}
    \end{enumerate}
\item\label{it:discrete_dif_3} $i=r>1$.
By interpreting as $x_{r+1}=1$ and
\[
\ii_{\bk_r^\wedge}^{(N)} (\bx_{r+1}^\wedge)\bigg|_{x_r+N^{-1}}=\ii^{(N)}_{(\bk_{r-1}^\downarrow)_r^\wedge}(\bx_{r+1}^\wedge)\bigg|_{x_r+N^{-1}}=0,
\]
formulas identical to \eqref{it:discrete_dif_2} hold.
\end{enumerate}
\end{theorem}
Difference equations obtained by replacing each $\ii_{\bk}^{(N)}(\bx)$ with $(-1)^{\dep(\bk)}\widetilde{\Li}_{\bk}^{\sh,(N)}(\bx)$ also hold (\cref{thm:diff_Li_1,thm:diff_Li_2}).
By independently proving the difference equations with respect to $x_1$ for each $\ii_{\bk}^{(N)}(\bx)$ and $\widetilde{\Li}_{\bk}^{\sh,(N)}(\bx)$, an alternative proof of \cref{thm:main} can be obtained.

This paper is organized as follows: In \cref{sec:connected_sums}, we provide a proof of \cref{thm:main} using the method of connected sums and prove that the limit of $\widetilde{\Li}_{\bk}^{\sh,(N)}(\bz)$ as $N\to\infty$ coincides with $\Li_{\bk}^{\sh}(\bz)$.
In \cref{sec:diff}, we prove \cref{thm:diff} and offer an alternative proof of the main theorem.
After the main theorem is proved, we explore several applications.
In \cref{sec:duality}, we present an alternative proof of the duality relations for multiple polylogarithms.
In \cref{sec:FMZV}, we investigate families of relations for finite multiple zeta values derived from our main result.
In \cref{sec:EDSR}, we give an alternative proof of the extended double shuffle relations for multiple polylogarithms.
In \cref{sec:Misc}, we exhibit a few equations obtained from our main result.
% --------------------------------------------------------------------------
\section{Proof of \cref{thm:main} by using the method of connected sums}\label{sec:connected_sums}
% --------------------------------------------------------------------------
In this section, we present a proof of \cref{thm:main} by using the method of connected sums as performed by Maesaka, Watanabe, and the third author~\cite{MaesakaSekiWatanabe-pre}.
See \cite{Seki2020} for the terms \emph{connector}, \emph{connected sum}, and \emph{transport relation}.

We aim to prove the following theorem, which is not \cref{thm:main} itself but a slightly modified version.
\begin{theorem}\label{thm:modified_main}
Let $\bk=(k_1,\dots, k_r)$ be an index and $\bx=(x_1,\dots, x_r)$ a tuple of indeterminates.
Then, for any positive integer $N$, we have
\begin{equation}\label{eq:modified_main}
\begin{split}
&\sum_{0<n_1<\cdots<n_r\leq N}\frac{1}{n_1^{k_1}\cdots n_r^{k_r}}\left[\prod_{i=1}^{r-1}\frac{\binom{Nx_{i+1}-1}{n_i}}{\binom{Nx_i-1}{n_i}}\right]\frac{\binom{N}{n_r}}{\binom{Nx_r-1}{n_r}}\\
&=(-1)^r\sum_{\substack{1\leq n_{j,1}\leq\cdots\leq n_{j,k_j}\leq N \ (1\leq j\leq r)\\ n_{j,k_j}<n_{j+1,1} \ (1\leq j<r)}}\prod_{j=1}^r\frac{1}{(n_{j,1}-Nx_j)n_{j,2}\cdots n_{j,k_j}}.
\end{split}
\end{equation}
\end{theorem}
% --------------------------------------------------------------------------
\subsection{Definition of the connector and the connected sum}
% --------------------------------------------------------------------------
Fix a positive integer $N$.
For an indeterminate $x$ and non-negative integers $n$ and $m$, we define a \emph{connector} $C_N^{(x)}(n,m)$ as
\[
C_N^{(x)}(n,m)\coloneqq\frac{\binom{m}{n}}{\binom{Nx-1}{n}}.
\]
For indices $\bk = (k_1, \dots, k_r)$ and $\bl = (l_1, \dots, l_s)$, let $S_N(\bk;\bl)$ be given by:
\begin{align*}
\left\{ (\bn, \bm_1, \ldots, \bm_s) \in \ZZ^r \times \ZZ^{l_1} \times \cdots \times \ZZ^{l_s} \ \middle|\ \begin{array}{l} 0 < n_1 < \cdots < n_r < m_{1,1},\\
m_{i,j} \leq N \quad (1 \le i \le s, 1 \le j \le l_i),\\
m_{i,j} \le m_{i,j+1} \quad (1 \le i \le s, 1 \le j < l_i),\\
m_{i, l_i} < m_{i+1, 1} \quad (1 \le i < s)
\end{array}\right\},
\end{align*}
where $\bn=(n_1,\dots, n_r)$ and $\bm_i=(m_{i,1},\dots, m_{i,l_i})$ ($1\leq i\leq s$).
Then, for a tuple of indeterminates $\bx=(x_1,\dots,x_{r+s})$, we define a \emph{connected sum} $Z_N^{(\bx)}(\bk \mid \bl)$ as
\[
Z_N^{(\bx)}(\bk \mid \bl) \coloneqq \sum_{(\bn, \bm_1, \ldots, \bm_s) \in S_N(\bk;\bl)} Q_{\bk}^{(x_1,\dots, x_r)}(\bn) \cdot C_N^{(x_r)}(n_r, m_{1,1}-1) \cdot \prod_{j=1}^s P_{N,l_j}^{(x_{r+j})}(\bm_j),
\]
where $Q_{\bk}^{(x_1,\dots, x_r)}(\bn)$ and $P_{N,l}^{(x)}(\bm)$ are defined by
\[
Q_{\bk}^{(x_1,\dots,x_r)}(\bn) \coloneqq \frac{1}{n_1^{k_1} \cdots n_r^{k_r}} \left[\prod_{i=1}^{r-1}\frac{\binom{Nx_{i+1}-1}{n_i}}{\binom{Nx_i-1}{n_i}}\right], \qquad P_{N,l}^{(x)}(\bm) \coloneqq \frac{1}{(Nx - m_1) m_2 \cdots m_l},
\]
respectively. Furthermore, we set
\[
Z_N^{(x_1,\dots, x_r)}(\bk \mid )\coloneqq\text{ L.H.S. of \eqref{eq:modified_main}},\quad Z_N^{(x_1,\dots,x_r)}( \mid \bk) \coloneqq\text{ R.H.S. of \eqref{eq:modified_main}}.
\]
% --------------------------------------------------------------------------
\subsection{Transport relations}
% --------------------------------------------------------------------------
Fix a positive integer $N$ and indeterminates $x$ and $x'$.
\begin{lemma}\label{lem:Trans}
For non-negative integers $n$ and $m$, we have
\begin{align}\label{eq:Trans-1}
\frac{1}{n} \cdot C_N^{(x)}(n,m) &= \sum_{n \le b \le m} C_N^{(x)}(n,b)\cdot\frac{1}{b} \quad (0 < n \le m),\\ \label{eq:Trans-2}
\sum_{n < a \le m} C_N^{(x)}(a,m)\cdot \frac{1}{m} &= C_N^{(x)}(n,m-1)\cdot\frac{1}{Nx-m} \quad (n < m),\\ \label{eq:Trans-3}
\frac{\binom{Nx'-1}{n}}{\binom{Nx-1}{n}}\cdot C_N^{(x')}(n,m-1) &= C_N^{(x)}(n,m-1) \quad (n < m). 
\end{align}
\end{lemma}
\begin{proof}
Since we can easily see that
\[
\frac{1}{n}\cdot\bigl(C_N^{(x)}(n,b) - C_N^{(x)}(n,b-1)\bigr) = C_N^{(x)}(n,b)\cdot\frac{1}{b}
\]
for $0<n<b\leq m$, and
\[
C_N^{(x)}(a,m)\cdot\frac{1}{m} = \bigl(C_N^{(x)}(a-1,m-1) - C_N^{(x)}(a,m-1)\bigr)\cdot\frac{1}{Nx-m}
\]
for $n<a\leq m$, we obtain the first two formulas. The last equation immediately follows by definition.
\end{proof}
By using \cref{lem:Trans}, we show the following transport relations.
\begin{lemma}\label{lem:Trans-rel}
Let $k$ be a positive integer, $\bk$ and $\bl$ indices.
Let $\bx$ be a tuple of indeterminates of appropriate length.
Then we have
\begin{align*}
Z_N^{(\bx)}(\bk, k \mid \bl) &= Z_N^{(\bx)}(\bk \mid k, \bl),\\
Z_N^{(\bx)}(\bk, k \mid ) &= Z_N^{(\bx)}(\bk \mid k),\\
Z_N^{(\bx)}(k \mid \bl) &= Z_N^{(\bx)}( \mid k, \bl).
\end{align*}
\end{lemma}
\begin{proof}
By repeatedly applying \eqref{eq:Trans-1} $k$ times, and subsequently using \eqref{eq:Trans-2} and \eqref{eq:Trans-3} once each, the conclusion can be obtained from the definitions of connected sums.
\end{proof}
% --------------------------------------------------------------------------
\subsection{Proof of \cref{thm:modified_main} and \cref{thm:main}}
% --------------------------------------------------------------------------
\begin{proof}[Proof of \cref{thm:modified_main}]
Write $\bk = (k_1, \dots, k_r)$.
By repeatedly using \cref{lem:Trans-rel} $r$ times, we have
\begin{align*}
\text{ L.H.S. of \eqref{eq:modified_main}} &= Z_N^{(\bx)}(k_1, \ldots, k_r \mid )\\
&= Z_N^{(\bx)}(k_1, \ldots, k_{r-1} \mid k_r)\\
&= \cdots\\
&= Z_N^{(\bx)}(\mid k_1, \dots, k_r) = \text{ R.H.S. of \eqref{eq:modified_main}}.
\end{align*}
This completes the proof.
\end{proof}
\begin{proof}[Proof of \cref{thm:main}]
It can be deduced from \cref{thm:modified_main} by replacing $x_j$ with $x_j(1+1/N)$ and then changing $N$ to $N-1$.
\end{proof}
The above proof is a straightforward generalization of the proof of \cref{thm:MSW} in \cite{MaesakaSekiWatanabe-pre}, but the process of changing the variables in the connector using \eqref{eq:Trans-3} is relatively new and interesting.
% --------------------------------------------------------------------------
\subsection{Behavior in the limit as $N\to\infty$}\label{subsec:CPWinS2}
% --------------------------------------------------------------------------
In this subsection, we check that
\[
\lim_{N\to\infty}\widetilde{\Li}_{\bk}^{\sh,(N)}(\bz)=\Li_{\bk}^{\sh}(\bz)
\]
holds.
\begin{lemma}\label{lem:Sum-Binom}
For distinct indeterminates $x$, $x'$ and integers $r$, $n$ with $0\leq r < n$, we have
\[
\sum_{r < i \le n} \frac{\binom{x'+1}{i}}{\binom{x}{i}} = \frac{x'+1}{x-x'} \left(\frac{\binom{x'}{r}}{\binom{x}{r}} - \frac{\binom{x'}{n}}{\binom{x}{n}}\right).
\]
In particular, for a fixed $z \in \CC$ with $|z| \ge 1$ and $z\neq 1$, 
\[
S_{N,n}^{(m, z)} \coloneqq \sum_{0 < i \le n} \frac{\binom{N-m}{i}}{\binom{Nz-m}{i}} = \frac{N-m}{Nz-N+1} \left(1 - \frac{\binom{N-m-1}{n}}{\binom{Nz-m}{n}}\right)
\]
is bounded for any integers $N, n, m$ satisfying $0 < m < N$ and $0 < n < N-m$.
\end{lemma}
\begin{proof}
The first claim immediately follows from
\[
\frac{\binom{x'+1}{i}}{\binom{x}{i}} = \frac{x'+1}{x-x'} \left(\frac{\binom{x'}{i-1}}{\binom{x}{i-1}} - \frac{\binom{x'}{i}}{\binom{x}{i}}\right).
\]
For the second claim, we observe that
\[
\left|\frac{\binom{N-m-1}{n}}{\binom{Nz-m}{n}} \right| = \prod_{j=1}^{n} \frac{N-m-j}{|Nz-m-j+1|} \le \prod_{j=1}^{n} \frac{N-m-j}{N-m-j+1} < 1
\]
by our assumption that $|z| \ge 1$. Since $(N-m)/(Nz-N+1)$ is bounded because $z\neq 1$, we obtain the desired result.
\end{proof}
\begin{proposition}\label{prop:asymptotic_for_modifiedMPL}
For an index $\bk=(k_1,\dots,k_r)$ and a tuple of indeterminates $\bz = (z_1, \dots, z_r)$ satisfying $|z_i| \ge 1$, where $(k_r,z_r)=(1,1)$ is also allowed, we have
\[
\widetilde{\Li}^{\sh,(N)}_{\bk}(\bz) = \Li^{\sh,< N}_{\bk}(\bz) + O(N^{-1/3}\log^rN)
\]
as $N \to \infty$.
The implied constant in Landau's notation depends only on $r$.
\end{proposition}
\begin{proof}
Take $0 \le h \le r$ such that $z_{h} \neq 1$ and $z_{h+1} = \cdots = z_r = (z_{r+1}=) 1$. In the case of $h = 0$, since $\widetilde{\Li}_{\bk}^{\sh, (N)} (\{1\}^r) = \Li_{\bk}^{\sh, <N}(\{1\}^r)$ holds, we assume $h > 0$ in the following. By a direct calculation, we obtain the expression
\[
\widetilde{\Li}_{\bk}^{\sh, (N)}(\bz) = \sum_{0 < n_1 < \cdots < n_r < N} \frac{1}{n_1^{k_1} \cdots n_r^{k_r}} \prod_{i=1}^r \left(\frac{z_{i+1}}{z_i}\right)^{n_i} \cdot \prod_{i=1}^{h} \frac{z_i^{n_i-n_{i-1}} \binom{N-n_{i-1}-1}{n_i-n_{i-1}}}{\binom{Nz_i-n_{i-1}-1}{n_i-n_{i-1}}},
\]
where $n_0 = 0$.
We decompose the sum into two parts based on the conditions $n_h<N^{1/3}$ or $N^{1/3}\leq n_h$.
Since for $n < n' < N^{1/3}$,
\begin{align*}
\frac{z^{n'-n} \binom{N-n-1}{n'-n}}{\binom{Nz-n-1}{n'-n}}
= \prod_{j=1}^{n'-n} \left(1+\frac{(1-z)(n+j)}{Nz-n-j} \right) = \prod_{j=1}^{n'-n} (1 + O(N^{-2/3}))
\end{align*}
holds, we have
\[
\prod_{i=1}^{h} \frac{z_i^{n_i-n_{i-1}} \binom{N-n_{i-1}-1}{n_i-n_{i-1}}}{\binom{Nz_i-n_{i-1}-1}{n_i-n_{i-1}}} = 1 + O_r(N^{-1/3})
\]
under the condition $n_h<N^{1/3}$.
Here, the subscript attached to $O$ indicates that the implied constant depends on that parameter.
Therefore, the first part equals
\begin{align*}
\sum_{\substack{0 < n_1 < \cdots < n_r < N\\ n_{h} < N^{1/3}}} \frac{1}{n_1^{k_1} \cdots n_r^{k_r}} \prod_{i=1}^r \left(\frac{z_{i+1}}{z_i}\right)^{n_i} (1 + O_r(N^{-1/3}))
\end{align*}
and the absolute value of the error term is bounded above by
\[
O_r\left(N^{-1/3}\sum_{0 < n_1, \ldots, n_r < N} \frac{1}{n_1 \cdots n_r}\right) = O_r(N^{-1/3} \log^r N)
\]
by our assumption $|z_i| \ge 1$.
Next, the second part is
\begin{align*}
&\sum_{0 < n_1 < \cdots < n_{h-1} < n_{h+1} < \cdots < n_r < N} \frac{1}{n_1^{k_1} \cdots n_{h-1}^{k_{h-1}} n_{h+1}^{k_{h+1}} \cdots n_r^{k_r}} \prod_{i=1}^{h-1} \frac{\binom{N-n_{i-1}-1}{n_i-n_{i-1}}}{\binom{Nz_i-n_{i-1}-1}{n_i-n_{i-1}}}\\
&\quad \times \sum_{\substack{n_{h-1}<n_h<n_{h+1} \\ N^{1/3}\leq n_h}} \frac{1}{n_{h}^{k_{h}}} \frac{\binom{N-n_{h-1}-1}{n_{h}-n_{h-1}}}{\binom{Nz_{h}-n_{h-1}-1}{n_{h}-n_{h-1}}}.
\end{align*}
(In the case of $h=r$, the description needs to be slightly modified, but the proof is entirely similar.)
Here, the inner sum is evaluated as
\begin{align*}
&\sum_{n'_{h-1} \le n_{h} < n_{h+1}} \frac{1}{n_{h}^{k_{h}}} \frac{\binom{N-n_{h-1}-1}{n_{h}-n_{h-1}}}{\binom{Nz_{h}-n_{h-1}-1}{n_{h}-n_{h-1}}}\\
&= \sum_{n'_{h-1} \le n_{h} < n_{h+1}} \frac{1}{n_{h}^{k_{h}}} (S_{N, n_h-n_{h-1}}^{(n_{h-1}+1, z_h)} - S_{N, n_h-n_{h-1}-1}^{(n_{h-1}+1, z_h)})\\
&= \sum_{n'_{h-1} \le n_{h} < n_{h+1}} \left(\frac{1}{n_h^{k_h}} - \frac{1}{(n_h+1)^{k_h}} \right) S_{N, n_h-n_{h-1}}^{(n_{h-1}+1, z_h)} + O_{z_h}(N^{-1/3})\\
&= O_{z_h}(N^{-1/3})
\end{align*}
by applying the boundedness of $S_{N,n}^{(m,z)}$ shown in \cref{lem:Sum-Binom}, where $n'_{h-1}\coloneqq \max\{n_{h-1}+1,N^{1/3}\}$ and $S_{N,0}^{(m,z)}\coloneqq 0$.
Since the remaining sum is $O(\log^{r-1}N)$ by $|z_i| \ge 1$, the second part is bounded above by $O(N^{-1/3}\log^r N)$.
In conclusion, we have
\begin{align*}
\widetilde{\Li}_{\bk}^{\sh, (N)}(\bz) = \sum_{\substack{0 < n_1 < \cdots < n_r < N\\ n_{h} < N^{1/3}}} \frac{1}{n_1^{k_1} \cdots n_r^{k_r}} \prod_{i=1}^r \left(\frac{z_{i+1}}{z_i}\right)^{n_i} + O_r(N^{-1/3} \log^r N).
\end{align*}
Finally, since
\[
\sum_{\substack{0 < n_1 < \cdots < n_r < N\\ N^{1/3} \le n_h}} \frac{1}{n_1^{k_1} \cdots n_r^{k_r}} \prod_{i=1}^r \left(\frac{z_{i+1}}{z_i}\right)^{n_i} = O(N^{-1/3} \log^r N)
\]
also holds by the method of Abel's summation, we obtain the desired result.
\end{proof}
From this proposition and the fact that $\ii_{\bk}^{(N)}(\bz)$ is a Riemann sum approximating $\ii_{\bk}(\bz)$, it can be considered that \cref{thm:main} indeed provides a discretization of \eqref{eq:IIEforMPL}.
% --------------------------------------------------------------------------
\section{Discretization of Goncharov's differential formula}\label{sec:diff}
% --------------------------------------------------------------------------
In this section, we fix a positive integer $N$, an index $\bk=(k_1,\dots, k_r)$, and a tuple of indeterminates $\bx=(x_1,\dots,x_r)$.
See \cref{sec:intro} for the definition of the difference quotient $\frac{\Delta^{(N)}}{\Delta^{(N)}x_i}$ and the abbreviated notation $\big|_{x_i+N^{-1}}$, $\bk^{\wedge}_{i}$, $\bk^{\downarrow}_i$, and $\bx^{\wedge}_i$.
% --------------------------------------------------------------------------
\subsection{Difference equations for $\widetilde{\Li}_{\bk}^{\sh,(N)}(\bx)$}
% --------------------------------------------------------------------------
\begin{theorem}\label{thm:diff_Li_1}
When $k_1>1$,
\[
\frac{\Delta^{(N)}\widetilde{\Li}_{\bk}^{\sh, (N)}(\bx)}{\Delta^{(N)}x_1}=-\frac{1}{x_1}\widetilde{\Li}_{\bk^{\downarrow}_1}^{\sh, (N)}(\bx);
\]
when $r=1$ and $k_1=1$, 
\[
\frac{\Delta^{(N)}\widetilde{\Li}_{\bk}^{\sh, (N)}(\bx)}{\Delta^{(N)}x_1}=\frac{1}{x_1}-\frac{1}{x_1+N^{-1}-1};
\]
when $r>1$ and $k_1=1$,
\[
\frac{\Delta^{(N)}\widetilde{\Li}_{\bk}^{\sh, (N)}(\bx)}{\Delta^{(N)}x_1} = \frac{1}{x_1}\widetilde{\Li}_{\bk^{\wedge}_1}^{\sh, (N)}(\bx^{\wedge}_1)-\frac{1}{x_1+N^{-1}-x_2}\left(\widetilde{\Li}_{\bk^{\wedge}_1}^{\sh, (N)}(\bx^{\wedge}_1)-\widetilde{\Li}_{\bk^{\wedge}_1}^{\sh, (N)}(\bx^{\wedge}_2)\bigg|_{x_1+N^{-1}}\right).
\]
\end{theorem}
\begin{proof}
Since
\[
\widetilde{\Li}_{\bk}^{\sh, (N)}(\bx) = \sum_{0 < n_1 < \cdots < n_r< N} \frac{1}{n_1^{k_1} \cdots n_r^{k_r}} \prod_{i=1}^r \frac{\binom{N-n_{i-1}-1}{n_i-n_{i-1}}}{\binom{Nx_i-n_{i-1}-1}{n_i-n_{i-1}}}
\]
with $n_0 = 0$, we have
\begin{align*}
\frac{\Delta^{(N)} \widetilde{\Li}_{\bk}^{\sh, (N)}(\bx)}{\Delta^{(N)} x_1} &= N\sum_{0 < n_1 < \cdots < n_r< N} \frac{1}{n_1^{k_1} \cdots n_r^{k_r}} \left(\frac{\binom{N-1}{n_1}}{\binom{Nx_1}{n_1}} - \frac{\binom{N-1}{n_1}}{\binom{Nx_1-1}{n_1}}\right) \prod_{i=2}^r \frac{\binom{N-n_{i-1}-1}{n_i-n_{i-1}}}{\binom{Nx_i-n_{i-1}-1}{n_i-n_{i-1}}}\\
&= -\frac{1}{x_1} \sum_{0 < n_1 < \cdots < n_r< N} \frac{1}{n_1^{k_1-1} n_2^{k_2} \cdots n_r^{k_r}} \prod_{i=1}^r \frac{\binom{N-n_{i-1}-1}{n_i-n_{i-1}}}{\binom{Nx_i-n_{i-1}-1}{n_i-n_{i-1}}}.
\end{align*}
If $k_1 > 1$, it equals $-x_1^{-1} \widetilde{\Li}_{\bk_1^{\downarrow}}^{\sh, (N)}(\bx)$.
If $r>1$ and $k_1 = 1$, by \cref{lem:Sum-Binom}, we compute
\begin{align*}
&-\frac{1}{x_1}\sum_{0 < n_1 < n_2} \frac{\binom{N-1}{n_1}}{\binom{Nx_1-1}{n_1}} \frac{\binom{N-n_1-1}{n_2-n_1}}{\binom{Nx_2-n_1-1}{n_2-n_1}} \\
&=-\frac{1}{x_1}\frac{\binom{N-1}{n_2}}{\binom{Nx_2-1}{n_2}} \sum_{0 < n_1 < n_2} \frac{\binom{Nx_2-1}{n_1}}{\binom{Nx_1-1}{n_1}}\\
&=\left(\frac{1}{x_1}-\frac{1}{x_1+N^{-1}-x_2}\right)\frac{\binom{N-1}{n_2}}{\binom{Nx_2-1}{n_2}}+\frac{1}{x_1+N^{-1}-x_2}\frac{\binom{N-1}{n_2}}{\binom{Nx_1}{n_2}}
\end{align*}
for each fixed $n_2$, and hence we have the desired result.
The case $r=1$ and $k_1=1$ is also calculated by using \cref{lem:Sum-Binom}.
\end{proof}
This calculation will be used later for an alternative proof of \cref{thm:main}.
On the other hand, for $i>1$, we only directly calculate $\frac{\Delta^{(N)}\ii_{\bk}^{(N)}(\bx)}{\Delta^{(N)}x_i}$, and derive the difference equations for $\frac{\Delta^{(N)}\widetilde{\Li}_{\bk}^{\sh,(N)}(\bx)}{\Delta^{(N)}x_i}$ as a consequence of combining it with \cref{thm:main}.
\begin{theorem}\label{thm:diff_Li_2}
We assume that $r>1$ and $i>1$.
When $k_{i-1}>1$ and $k_i>1$,
\[
\frac{\Delta^{(N)}\widetilde{\Li}_{\bk}^{\sh,(N)}(\bx)}{\Delta^{(N)}x_i} = \frac{1}{x_i}\left(\widetilde{\Li}^{\sh,(N)}_{\bk^{\downarrow}_{i-1}}(\bx)\bigg|_{x_i+N^{-1}}-\widetilde{\Li}^{\sh,(N)}_{\bk^{\downarrow}_i}(\bx)\right);
\]
when $k_{i-1}>1$ and $k_i=1$,
\begin{align*}
\frac{\Delta^{(N)} \widetilde{\Li}_{\bk}^{\sh,(N)}(\bx)}{\Delta^{(N)} x_i} &= \frac{1}{x_i}\left(\widetilde{\Li}_{\bk_{i-1}^\downarrow}^{\sh,(N)}(\bx)\bigg|_{x_i+N^{-1}}+\widetilde{\Li}_{\bk_i^\wedge}^{\sh,(N)}(\bx_i^\wedge)\right)\\
&\quad-\frac{1}{x_i+N^{-1}-x_{i+1}}\left(\widetilde{\Li}_{\bk_i^\wedge}^{\sh,(N)}(\bx_i^\wedge) - \widetilde{\Li}_{\bk_i^\wedge}^{\sh,(N)}(\bx_{i+1}^\wedge)\bigg|_{x_i+N^{-1}}\right)\\
&\quad-\frac{1}{N}\frac{1}{x_i(x_i+N^{-1}-x_{i+1})}\left(\widetilde{\Li}_{(\bk_{i-1}^\downarrow)_i^\wedge}^{\sh,(N)}(\bx_{i+1}^\wedge)\bigg|_{x_i+N^{-1}} - \widetilde{\Li}_{(\bk_{i-1}^\downarrow)_i^\wedge}^{\sh,(N)}(\bx_i^\wedge)\right);
\end{align*}
when $k_{i-1}=1$ and $k_i>1$,
\begin{align*}
\frac{\Delta^{(N)}\widetilde{\Li}_{\bk}^{\sh,(N)}(\bx)}{\Delta^{(N)}x_i}&=-\frac{1}{x_i-x_{i-1}} \left(\widetilde{\Li}_{\bk_{i-1}^\wedge}^{\sh,(N)}(\bx_{i-1}^\wedge) - \widetilde{\Li}_{\bk_{i-1}^\wedge}^{\sh,(N)}(\bx_i^\wedge) \right) \\
&\quad- \frac{1}{x_i} \left(\widetilde{\Li}_{\bk_{i-1}^\wedge}^{\sh,(N)}(\bx_i^\wedge) + \widetilde{\Li}_{\bk_i^\downarrow}^{\sh,(N)}(\bx)\right)\\
&\quad-\frac{1}{N}\frac{1}{x_i(x_i-x_{i-1})}\left(\widetilde{\Li}_{(\bk^{\downarrow}_i)^{\wedge}_{i-1}}^{\sh,(N)}(\bx^{\wedge}_i)-\widetilde{\Li}_{(\bk^{\downarrow}_i)^{\wedge}_{i-1}}^{\sh,(N)}(\bx^{\wedge}_{i-1})\right);
\end{align*}
when $k_{i-1}=k_i=1$,
\begin{align*}
\frac{\Delta^{(N)} \widetilde{\Li}_{\bk}^{\sh,(N)}(\bx)}{\Delta^{(N)} x_i} &= -\frac{1}{x_i - x_{i-1}} \left(\widetilde{\Li}_{\bk_{i-1}^\wedge}^{\sh,(N)}(\bx_{i-1}^\wedge) - \widetilde{\Li}_{\bk_i^\wedge}^{\sh,(N)} (\bx_i^\wedge)\right) \\
&\quad- \frac{1}{x_i+N^{-1}-x_{i+1}} \left(\widetilde{\Li}_{\bk_i^\wedge}^{\sh,(N)}(\bx_i^\wedge) - \widetilde{\Li}_{\bk_i^\wedge}^{\sh,(N)} (\bx_{i+1}^\wedge)\bigg|_{x_i+N^{-1}} \right).
\end{align*}
For the case $i=r$, we should interpret as $x_{r+1}=1$ and 
\[
\widetilde{\Li}_{\bk_r^\wedge}^{\sh,(N)}(\bx_{r+1}^\wedge)\bigg|_{x_r+N^{-1}} = \widetilde{\Li}_{(\bk_{r-1}^\downarrow)_r^\wedge}^{\sh,(N)}(\bx_{r+1}^\wedge)\bigg|_{x_r+N^{-1}}=0.
\]
\end{theorem}
\begin{proof}
This follows from \cref{thm:main} and \cref{thm:diff_ii_2_1,thm:diff_ii_2_2,thm:diff_ii_2_3,thm:diff_ii_2_4}.
\end{proof}
% --------------------------------------------------------------------------
\subsection{Difference equations for $\ii_{\bk}^{(N)}(\bx)$}
% --------------------------------------------------------------------------
\begin{theorem}\label{thm:diff_ii_1}
When $k_1>1$,
\[
\frac{\Delta^{(N)}\ii_{\bk}^{(N)}(\bx)}{\Delta^{(N)}x_1}=-\frac{1}{x_1}\ii_{\bk^{\downarrow}_1}^{(N)}(\bx);
\]
when $r=1$ and $k_1=1$, 
\[
\frac{\Delta^{(N)}\ii_{\bk}^{(N)}(\bx)}{\Delta^{(N)}x_1}=\frac{1}{x_1+N^{-1}-1}-\frac{1}{x_1};
\]
when $r>1$ and $k_1=1$,
\[
\frac{\Delta^{(N)}\ii_{\bk}^{(N)}(\bx)}{\Delta^{(N)}x_1}=-\frac{1}{x_1}\ii_{\bk^{\wedge}_1}^{(N)}(\bx^{\wedge}_1)+\frac{1}{x_1+N^{-1}-x_2}\left(\ii_{\bk^{\wedge}_1}^{(N)}(\bx^{\wedge}_1)-\ii_{\bk^{\wedge}_1}^{(N)}(\bx^{\wedge}_2)\bigg|_{x_1+N^{-1}}\right).
\]
\end{theorem}
\begin{proof}
By definition, we have
\begin{align*}
\frac{\Delta^{(N)} \ii_{\bk}^{(N)}(\bx)}{\Delta^{(N)} x_1} &= N\sum_{\substack{0< n_{j,1}\leq\cdots\leq n_{j,k_j}< N \ (1\leq j\leq r)\\ n_{j,k_j}<n_{j+1,1} \ (1\leq j<r)}} \left(\frac{1}{n_{1,1} - 1 - Nx_1} - \frac{1}{n_{1,1} - Nx_1} \right)\\
&\quad \times \frac{1}{n_{1,2} \cdots n_{1,k_1}} \prod_{j=2}^r\frac{1}{(n_{j,1}-Nx_j)n_{j,2}\cdots n_{j,k_j}}.
\end{align*}
If $k_1 > 1$, then for each fixed $n_{1,2}$, we have
\[
N\sum_{0 < n_{1,1} \le n_{1,2}} \left(\frac{1}{n_{1,1} - 1 - Nx_1} - \frac{1}{n_{1,1} - Nx_1} \right) = -\frac{n_{1,2}}{x_1}\cdot\frac{1}{n_{1,2}-Nx_1},
\]
which implies that
\[
\frac{\Delta^{(N)} \ii_{\bk}^{(N)}(\bx)}{\Delta^{(N)} x_1} = -\frac{1}{x_1} \ii_{\bk^{\downarrow}_1}^{(N)}(\bx).
\]
The case $r=1$ and $k_1=1$ is easy.
If $r>1$ and $k_1 = 1$, then for each fixed $n_{2,1}$, we have
\begin{align*}
&N\sum_{0 < n_{1,1} < n_{2,1}} \left(\frac{1}{n_{1,1} - 1 - Nx_1} - \frac{1}{n_{1,1} - Nx_1} \right) \frac{1}{n_{2,1}-Nx_2}\\
&=-\frac{1}{x_1}\cdot\frac{1}{n_{2,1}-Nx_2}+\frac{1}{x_1+N^{-1}-x_2}\left(\frac{1}{n_{2,1}-Nx_2}-\frac{1}{n_{2,1}-N(x_1+N^{-1})}\right),
\end{align*}
which implies the desired result.
\end{proof}
\begin{theorem}\label{thm:diff_ii_2_1}
We assume that $r>1$ and $i>1$.
When $k_{i-1}>1$ and $k_i>1$,
\[
\frac{\Delta^{(N)}\ii_{\bk}^{(N)}(\bx)}{\Delta^{(N)}x_i} = \frac{1}{x_i}\left(\ii^{(N)}_{\bk^{\downarrow}_{i-1}}(\bx)\bigg|_{x_i+N^{-1}}-\ii^{(N)}_{\bk^{\downarrow}_i}(\bx)\right).
\]
\end{theorem}
\begin{proof}
This follows from
\begin{align*}
&N\sum_{n_{i-1, k_{i-1}-1} \le n_{i-1, k_{i-1}} < n_{i,1} \le n_{i,2}} \frac{1}{n_{i-1, k_{i-1}}} \left(\frac{1}{n_{i,1} - 1 - Nx_i} - \frac{1}{n_{i,1}-Nx_i}\right) \frac{1}{n_{i,2}}\\
&= N\sum_{n_{i-1, k_{i-1}-1} \le n_{i-1, k_{i-1}} < n_{i,2}} \frac{1}{n_{i-1, k_{i-1}}} \left(\frac{1}{n_{i-1,k_{i-1}} - Nx_i} - \frac{1}{n_{i,2}-Nx_i}\right) \frac{1}{n_{i,2}}\\
&= \frac{1}{x_i}\sum_{n_{i-1, k_{i-1}-1} \le n_{i-1, k_{i-1}} < n_{i,2}} \left(\frac{1}{(n_{i-1,k_{i-1}} - Nx_i) n_{i,2}} - \frac{1}{n_{i-1,k_{i-1}} (n_{i,2} - Nx_i)} \right)
\end{align*}
and
\begin{align*}
&\sum_{n_{i-1, k_{i-1}-1} \le n_{i-1, k_{i-1}} < n_{i,2}} \frac{1}{n_{i-1,k_{i-1}} - Nx_i}\cdot\frac{1}{n_{i,2}} \\
&= \sum_{n_{i-1, k_{i-1}-1} < n_{i,1} \le n_{i,2}} \frac{1}{n_{i,1} - N(x_i+N^{-1})}\cdot\frac{1}{n_{i,2}}.\qedhere
\end{align*}
\end{proof}
\begin{theorem}\label{thm:diff_ii_2_2}
We assume that $r>1$.
When $1<i<r$, $k_{i-1}>1$, and $k_i=1$,
\begin{align*}
\frac{\Delta^{(N)} \ii_{\bk}^{(N)}(\bx)}{\Delta^{(N)} x_i} &= \frac{1}{x_i}\left(\ii_{\bk_{i-1}^\downarrow}^{(N)}(\bx)\bigg|_{x_i+N^{-1}}-\ii_{\bk_i^\wedge}^{(N)}(\bx_i^\wedge)\right)\\
&\quad+\frac{1}{x_i+N^{-1}-x_{i+1}}\left(\ii_{\bk_i^\wedge}^{(N)}(\bx_i^\wedge) - \ii_{\bk_i^\wedge}^{(N)}(\bx_{i+1}^\wedge)\bigg|_{x_i+N^{-1}}\right)\\
&\quad+\frac{1}{N}\frac{1}{x_i(x_i+N^{-1}-x_{i+1})}\left(\ii^{(N)}_{(\bk_{i-1}^\downarrow)_i^\wedge}(\bx_{i+1}^\wedge)\bigg|_{x_i+N^{-1}} - \ii^{(N)}_{(\bk_{i-1}^\downarrow)_i^\wedge}(\bx_i^\wedge)\right);
\end{align*}
when $k_{r-1}>1$ and $k_r=1$,
\begin{align*}
\frac{\Delta^{(N)} \ii_{\bk}^{(N)}(\bx)}{\Delta^{(N)} x_r} &= \frac{1}{x_r}\ii_{\bk_{r-1}^\downarrow}^{(N)}(\bx)\bigg|_{x_r+N^{-1}}+\left(\frac{1}{x_r+N^{-1}-1}-\frac{1}{x_r}\right)\ii_{\bk_r^\wedge}^{(N)}(\bx_r^\wedge)\\
&\quad-\frac{1}{N}\frac{1}{x_r(x_r+N^{-1}-1)}\ii^{(N)}_{(\bk_{r-1}^\downarrow)_r^\wedge}(\bx_r^\wedge).
\end{align*}
\end{theorem}
\begin{proof}
The case $i<r$ follows from
\begin{align*}
&N\sum_{n_{i-1, k_{i-1}-1} \le n_{i-1, k_{i-1}} < n_{i,1} < n_{i+1, 1}} \frac{1}{n_{i-1, k_{i-1}}} \\
&\qquad\qquad\times\left(\frac{1}{n_{i,1}-1-Nx_i} - \frac{1}{n_{i,1}-Nx_i} \right) \frac{1}{n_{i+1,1} - Nx_{i+1}}\\
&= N \sum_{n_{i-1, k_{i-1}-1} \le n_{i-1, k_{i-1}} < n_{i+1, 1}} \frac{1}{n_{i-1, k_{i-1}}} \\
&\qquad\qquad \times \left(\frac{1}{n_{i-1, k_{i-1}} - Nx_i} - \frac{1}{n_{i+1,1} - 1-Nx_i}\right) \frac{1}{n_{i+1,1}-Nx_{i+1}}\\
&= \sum_{n_{i-1, k_{i-1}-1} \le n_{i-1, k_{i-1}} < n_{i+1, 1}} \left[\frac{1}{x_i} \left(\frac{1}{n_{i-1, k_{i-1}} -Nx_i} - \frac{1}{n_{i-1, k_{i-1}}}\right) \frac{1}{n_{i+1,1}-Nx_{i+1}}\right.\\
&\qquad\qquad + \left.\frac{1}{x_i +N^{-1} - x_{i+1}}\cdot\frac{1}{n_{i-1, k_{i-1}}} \left(\frac{1}{n_{i+1,1}-Nx_{i+1}} - \frac{1}{n_{i+1,1}-N(x_i+N^{-1})} \right) \right]
\end{align*}
and
\begin{align*}
&\sum_{n_{i-1,k_{i-1}-1} \le n_{i-1, k_{i-1}} < n_{i+1,1}} \frac{1}{n_{i-1,k_{i-1}}-Nx_i}\cdot \frac{1}{n_{i+1,1}-Nx_{i+1}}\\
&= \sum_{n_{i-1,k_{i-1}-1} < n_{i,1} \le n_{i+1,1}} \frac{1}{n_{i,1}-N(x_i+N^{-1})} \cdot \frac{1}{n_{i+1,1}-Nx_{i+1}}\\
&= \sum_{n_{i-1,k_{i-1}-1} < n_{i,1} < n_{i+1,1}} \frac{1}{n_{i,1}-N(x_i+N^{-1})} \cdot \frac{1}{n_{i+1,1}-Nx_{i+1}}\\
&\quad+ \frac{1}{N}\cdot\frac{1}{x_i+N^{-1}-x_{i+1}}\left(\frac{1}{n_{i+1,1}-N(x_i+N^{-1})}-\frac{1}{n_{i+1,1}-Nx_{i+1}}\right).
\end{align*}
The cases $i=r$ is similar.
\end{proof}
\begin{theorem}\label{thm:diff_ii_2_3}
We assume that $r>1$ and $i>1$.
When $k_{i-1}=1$ and $k_i>1$,
\begin{align*}
\frac{\Delta^{(N)}\ii_{\bk}^{(N)}(\bx)}{\Delta^{(N)}x_i}&=\frac{1}{x_i-x_{i-1}} \left(\ii_{\bk_{i-1}^\wedge}^{(N)}(\bx_{i-1}^\wedge) - \ii_{\bk_{i-1}^\wedge}^{(N)}(\bx_i^\wedge) \right) + \frac{1}{x_i} \left(\ii_{\bk_{i-1}^\wedge}^{(N)}(\bx_i^\wedge) - \ii_{\bk_i^\downarrow}^{(N)}(\bx)\right)\\
&\quad+\frac{1}{N}\frac{1}{x_i(x_i-x_{i-1})}\left(\ii_{(\bk^{\downarrow}_i)^{\wedge}_{i-1}}^{(N)}(\bx^{\wedge}_i)-\ii_{(\bk^{\downarrow}_i)^{\wedge}_{i-1}}^{(N)}(\bx^{\wedge}_{i-1})\right).
\end{align*}
\end{theorem}
\begin{proof}
This follows from
\begin{align*}
&N\sum_{n_{i-2,k_{i-2}}<n_{i-1,1}<n_{i,1}\leq n_{i,2}}\frac{1}{n_{i-1,1}-Nx_{i-1}}\left(\frac{1}{n_{i,1}-1-Nx_i}-\frac{1}{n_{i,1}-Nx_i}\right)\frac{1}{n_{i,2}}\\
&=N\sum_{n_{i-2,k_{i-2}}<n_{i-1,1}<n_{i,2}}\frac{1}{n_{i-1,1}-Nx_{i-1}}\left(\frac{1}{n_{i-1,1}-Nx_i}-\frac{1}{n_{i,2}-Nx_i}\right)\frac{1}{n_{i,2}}\\
&=\sum_{n_{i-2,k_{i-2}}<n_{i-1,1}<n_{i,2}}\left[\frac{1}{x_i-x_{i-1}}\left(\frac{1}{n_{i-1,1}-Nx_i}-\frac{1}{n_{i-1,1}-Nx_{i-1}}\right)\frac{1}{n_{i,2}}\right.\\
&\qquad\qquad\left.+\frac{1}{x_i}\cdot\frac{1}{n_{i-1,1}-Nx_{i-1}}\left(\frac{1}{n_{i,2}}-\frac{1}{n_{i,2}-Nx_i}\right)\right]
\end{align*}
and
\begin{align*}
&\frac{1}{x_i-x_{i-1}}\left(\frac{1}{n_{i,2}-Nx_i}-\frac{1}{n_{i,2}-Nx_{i-1}}\right)\frac{1}{n_{i,2}}+\frac{1}{x_i}\cdot\frac{1}{n_{i,2}-Nx_{i-1}}\cdot\frac{1}{n_{i,2}}\\
&=\frac{1}{N}\frac{1}{x_i(x_i-x_{i-1})}\left(\frac{1}{n_{i,2}-Nx_i}-\frac{1}{n_{i,2}-Nx_{i-1}}\right).\qedhere
\end{align*}
\end{proof}
\begin{theorem}\label{thm:diff_ii_2_4}
We assume that $r>1$.
When $1<i<r$ and $k_{i-1}=k_i=1$,
\begin{align*}
\frac{\Delta^{(N)} \ii_{\bk}^{(N)}(\bx)}{\Delta^{(N)} x_i} &= \frac{1}{x_i - x_{i-1}} \left(\ii_{\bk_{i-1}^\wedge}^{(N)}(\bx_{i-1}^\wedge) - \ii_{\bk_i^\wedge}^{(N)} (\bx_i^\wedge)\right) \\
&\quad+ \frac{1}{x_i+N^{-1}-x_{i+1}} \left(\ii_{\bk_i^\wedge}^{(N)}(\bx_i^\wedge) - \ii_{\bk_i^\wedge}^{(N)} (\bx_{i+1}^\wedge)\bigg|_{x_i+N^{-1}} \right);
\end{align*}
when $k_{r-1}=k_r=1$,
\[
\frac{\Delta^{(N)} \ii_{\bk}^{(N)}(\bx)}{\Delta^{(N)} x_r} = \frac{1}{x_r - x_{r-1}} \left(\ii_{\bk_{r-1}^\wedge}^{(N)}(\bx_{r-1}^\wedge) - \ii_{\bk_r^\wedge}^{(N)} (\bx_r^\wedge)\right)+\frac{1}{x_r+N^{-1}-1}\ii_{\bk^{\wedge}_r}^{(N)}(\bx^{\wedge}_r).
\]
\end{theorem}
\begin{proof}
This follows from
\begin{align*}
&N\sum_{n_{i-2,k_{i-2}}<n_{i-1,1}<n_{i,1}<n_{i+1,1}}\frac{1}{n_{i-1,1}-Nx_{i-1}}\\
&\qquad\qquad\times\left(\frac{1}{n_{i,1}-1-Nx_i}-\frac{1}{n_{i,1}-Nx_i}\right)\frac{1}{n_{i+1,1}-Nx_{i+1}}\\
&=N\sum_{n_{i-2,k_{i-2}}<n_{i-1,1}<n_{i+1,1}}\frac{1}{n_{i-1,1}-Nx_{i-1}}\\
&\qquad\qquad\times\left(\frac{1}{n_{i-1,1}-Nx_i}-\frac{1}{n_{i+1,1}-N(x_i+N^{-1})}\right)\frac{1}{n_{i+1,1}-Nx_{i+1}}\\
&=\sum_{n_{i-2,k_{i-2}}<n_{i-1,1}<n_{i+1,1}}\left[\frac{1}{x_i-x_{i-1}}\left(\frac{1}{n_{i-1,1}-Nx_i}-\frac{1}{n_{i-1,1}-Nx_{i-1}}\right)\frac{1}{n_{i+1,1}-Nx_{i+1}}\right.\\
&\qquad+\left.\frac{1}{x_i+N^{-1}-x_{i+1}}\cdot\frac{1}{n_{i-1,1}-Nx_{i-1}}\left(\frac{1}{n_{i+1,1}-Nx_{i+1}}-\frac{1}{n_{i+1,1}-N(x_i+N^{-1})}\right)\right]
\end{align*}
for the case $i<r$.
The case $i=r$ is similar.
\end{proof}
% --------------------------------------------------------------------------
\subsection{Proof of \cref{thm:main} by using the difference equations}
% --------------------------------------------------------------------------
We use the difference equations with respect to $x_1$.
\begin{proof}[An alternating proof of \cref{thm:main}]
We prove \eqref{eq:main} by induction on $\wt(\bk)$.
By \cref{thm:diff_Li_1,thm:diff_ii_1}, and the induction hypothesis, we have
\[
\frac{\Delta^{(N)}\widetilde{\Li}_{\bk}^{\sh,(N)}(\bx)}{\Delta^{(N)}x_1}=(-1)^{\dep(\bk)}\cdot\frac{\Delta^{(N)}\ii_{\bk}^{(N)}(\bx)}{\Delta^{(N)}x_1}.
\]
For the case where $\bk=(1)$, it holds without any assumptions.
In particular, by setting $F(x_1)\coloneqq\widetilde{\Li}_{\bk}^{\sh,(N)}(\bx)-(-1)^{\dep(\bk)}\ii_{\bk}^{(N)}(\bx)$, we see that $F(x_1+N^{-1})=F(x_1)$ holds.
By definition, one can decompose $F(x_1)$ as
\[
F(x_1)=\sum_{n=1}^{N-1}\frac{C_n(x_2,\dots,x_r)}{Nx_1-n}.
\]
If there exists $n$ such that $C_n\neq 0$, then $F(x_1)$ has a pole at $x_1=0$ by $F(x_1+n/N)=F(x_1)$.
However, it is impossible by the definition of $\widetilde{\Li}_{\bk}^{\sh,(N)}(\bx)$ and $\ii_{\bk}^{(N)}(\bx)$.
Hence we have all $C_n=0$ and $F(x_1)=0$.
\end{proof}
% --------------------------------------------------------------------------
\section{Duality for multiple polylogarithms}\label{sec:duality}
% --------------------------------------------------------------------------
In \cite{MaesakaSekiWatanabe-pre}, a new proof of the duality for MZVs is provided through the manipulation of finite sums as an application of \cref{thm:MSW}.
Here, we present a similar proof of the duality for MPLs as an application of \cref{thm:main}.
% --------------------------------------------------------------------------
\subsection{Notation and the statement}
% --------------------------------------------------------------------------
Any admissible index $\bk \neq \emp$ is uniquely expressed as $\bk = (\{1\}^{a_1-1}, b_1+1, \ldots, \{1\}^{a_h-1}, b_h+1)$, where $h, a_1, \ldots, a_h, b_1, \ldots, b_h$ are positive integers.
Then, its \emph{dual index} $\bk^{\dagger}$ is defined as $\bk^\dagger \coloneqq (\{1\}^{b_h-1}, a_h+1, \ldots, \{1\}^{b_1-1}, a_1+1)$.
Here, we also consider the empty index $\emp$ as an admissible index, and set $\emp^{\dagger}\coloneqq\emp$, $\dep(\emp)\coloneqq0$.
The symbol $\BB$ denotes $\{z\in\CC \mid |z|\geq 1, |1-z|\geq 1\}\cup\{1\}$.
We say that a pair $\left({\bz \atop \bk}\right)$, consisting of an index $\bk\neq\emp$ with $\dep(\bk)=r$ and a tuple of complex numbers $\bz=(z_1,\dots,z_r)$, satisfies the \emph{dual condition} if $z_i\in\BB$ for all $1\leq i\leq r$, and additionally, if $\bk$ is non-admissible, then it is required that $z_r\neq 1$. A pair $\left({\bz \atop \bk}\right)$ satisfying the dual condition is uniquely expressed as
\[
\left({\bz\atop\bk}\right)=\left({\{1\}^{r_1},\atop\bl_1,}{\{1\}^{a_1-1},\atop\{1\}^{a_1-1},}{w_1\atop b_1}{,\dots,\atop,\dots,}{\{1\}^{r_d},\atop\bl_d,}{\{1\}^{a_d-1},\atop\{1\}^{a_d-1},}{w_d,\atop b_d,}{\{1\}^{r_{d+1}}\atop\bl_{d+1}}\right),
\]
where $d$ is a non-negative integer, $a_1, \ldots, a_d, b_1, \ldots, b_d$ are positive integers, all $w_1, \ldots, w_d$ are not $1$, $\bl_1, \ldots, \bl_{d+1}$ are admissible indices, and $r_j \coloneqq \mathrm{dep}(\bl_j)$ for $1 \le j \le d+1$.
Then its dual pair $\left({\bz \atop \bk} \right)^\dagger$ is defined by
\[
\left({\bz\atop\bk}\right)^{\dagger}=\left({\{1\}^{s_{d+1}},\atop(\bl_{d+1})^{\dagger},}{\{1\}^{b_d-1},\atop\{1\}^{b_d-1},}{1-w_d,\atop a_d,}{\{1\}^{s_d}\atop(\bl_d)^{\dagger}}{,\dots,\atop,\dots,}{\{1\}^{b_1-1},\atop\{1\}^{b_1-1},}{1-w_1,\atop a_1,}{\{1\}^{s_{1}}\atop(\bl_{1})^{\dagger}}\right),
\]
where $s_j \coloneqq \mathrm{dep}(\bl_j^\dagger)$ for $1 \le j \le d+1$.
For $\bz$, set $\iota(\bz)\coloneqq d$.
\begin{theorem}[Duality for MPLs {\cite[Section~6.1]{BorweinEtal2001}, \cite[Theorem~3.4]{KawamuraMaesakaSeki2022}}]\label{thm:duality_for_MPL}
Let $\left({\bz \atop \bk}\right)$ be a pair satisfying the dual condition.
Write $\left({\bz \atop \bk}\right)^{\dagger}$ as $\left({\bz' \atop \bk'}\right)$.
Then we have
\[
\Li_{\bk}^{\sh}(\bz)=(-1)^{\iota(\bz)}\Li_{\bk'}^{\sh}(\bz').
\]
\end{theorem}
In \cite{BorweinEtal2001}, this theorem was proved using the iterated integral expression \eqref{eq:IIEforMPL}, but a proof via series manipulation was considered to be difficult.
In contrast, \cite{KawamuraMaesakaSeki2022} successfully provided an alternative proof through the manipulation of infinite series\footnote{Note that the definitions of multiple  polylogarithms are slightly different; their $\Li_{\bk}^{\sh}(z_1,\dots,z_r)$ corresponds to our $\Li_{\bk}^{\sh}(z_1^{-1},\dots,z_r^{-1})$.}.
Following \cite{MaesakaSekiWatanabe-pre}, we present a proof by manipulating finite sums as an application of our main theorem.
While \cite{KawamuraMaesakaSeki2022} excludes the case of conditional convergence, here, we include and discuss that case as well.
% --------------------------------------------------------------------------
\subsection{Error estimates}
% --------------------------------------------------------------------------
We prepare a lemma on the necessary error estimates, including those used in \cref{sec:EDSR}.
\begin{definition}\label{def:R-values}
For positive integers $N$, $k$, non-negative integers $a_1, b_1, \ldots, a_k, b_k$ satisfying all $a_i+b_i\geq 1$, and a tuple of complex numbers
\[
\bz=(z_{1,1},\dots,z_{1,a_1}, z_{2,1},\dots, z_{2,a_2}, \dots,z_{k,1},\dots,z_{k,a_k})\in\CC^{a_1+\cdots+a_k}
\]
satisfying all $|z_{i,j}|\geq 1$, we set
\begin{align*}
R_{<N}(a_1,\dots,a_k;b_1,\dots, b_k)&\coloneqq\sum_{0 < n_1 < \cdots < n_k < N} \prod_{i=1}^k \frac{1}{(N-n_i)^{a_i} n_i^{b_i}}, \\
R_{<N}^{(\bz)}(a_1,\dots,a_k;b_1,\dots, b_k)&\coloneqq\sum_{0<n_1<\cdots<n_k<N}\prod_{i=1}^k\frac{1}{(n_i-Nz_{i,1})\cdots (n_i-Nz_{i,a_i})n_i^{b_i}}.
\end{align*}
\end{definition}
\begin{lemma}\label{lem:error}
In the setting of \cref{def:R-values}, we have
\[
R_{<N}^{(\bz)}(a_1,\dots,a_k;b_1,\dots, b_k)=O(\log^k N)
\]
as $N\to\infty$.
Furthermore, assuming that $a_1\geq 1$, we assume that one of the following three conditions is satisfied$:$
\begin{itemize}
\item there exists at least one $1\leq i\leq k$ satisfying $a_i+b_i\geq 2$ and $b_i\geq 1$.
\item there exist $1\leq i<j\leq k$ satisfying $a_i\geq 2$ and $b_j\geq 1$.
\item there exist $1\leq i\leq j\leq k$, $1\leq l\leq a_j$ satisfying $a_i\geq 2$ and $z_{j,l}\neq 1$.
\end{itemize}
Then we have
\[
R_{<N}^{(\bz)}(a_1,\dots,a_k;b_1,\dots, b_k)=O(N^{-1}\log^k N)
\]
as $N\to\infty$.
The implied constant in Landau's notation depends on $z_{j,l}$ under the third condition.
\end{lemma}
\begin{proof}
Since
\[
|R_{<N}^{(\bz)}(a_1,\dots,a_k;b_1,\dots, b_k)|\leq R_{<N}(a_1,\dots,a_k;b_1,\dots, b_k)
\]
holds, except for the case of the last condition, it follows from \cite[Lemma~2.1]{MaesakaSekiWatanabe-pre} and \cite[Lemma~2.2]{Seki-pre}.
To avoid cumbersome notation, we will only prove a simple case where the last condition is satisfied.
Let $z_1$, $z_2$, and $w$ be complex numbers satisfying $|z_1|\geq 1$, $|z_2|\geq 1$, $|w|\geq 1$, and $w\neq 1$.
Then, 
\begin{align*}
|R_{<N}^{(z_1,z_2,w)}(2,1;0,0)|&\leq \sum_{0<n<m<N}\frac{1}{(N-n)^2|Nw-m|}\\
&=\sum_{0<m'<n'<N}\frac{1}{|N(1-w)-m'|(n')^2}\\
&<\sum_{0<m'<n'<N}\frac{1}{|N(1-w)-m'|m'n'}\\
&\leq\frac{1}{N|1-w|}\left(\sum_{0<m'<n'<N}\frac{1}{|N(1-w)-m'|n'}+\sum_{0<m'<n'<N}\frac{1}{m'n'}\right)\\
&=\frac{1}{N|1-w|}\left(\sum_{0<n<m<N}\frac{1}{(N-n)|Nw-m|}+R_{<N}(1,1;0,0)\right)\\
&\leq\frac{2R_{<N}(1,1;0,0)}{N|1-w|}=O_w(N^{-1}\log^2N)
\end{align*}
as $N\to\infty$.
The general case can be proved in exactly the same manner.
\end{proof}
% --------------------------------------------------------------------------
\subsection{Proof of \cref{thm:duality_for_MPL}}
% --------------------------------------------------------------------------
Let $\left({\bz \atop \bk}\right)$ and $\left({\bz' \atop \bk'}\right)$ as in the statement of \cref{thm:duality_for_MPL}.
Note that $\wt(\bk)=\dep(\bk)+\dep(\bk')-\iota(\bz)$.
\begin{theorem}[Asymptotic duality]\label{thm:asymp_duality}
\[
\ii_{\bk'}^{(N)}(\bz')=(-1)^{\wt(\bk)}\ii_{\bk}^{(N)}(\bz)+O(N^{-1}\log^{\wt(\bk)}N)
\]
as $N\to\infty$.
\end{theorem}
\begin{proof}
After applying the change of variables ``$n_i\mapsto N-n_{\wt(\bk)+1-i}$'' in the definition of $\ii_{\bk'}^{(N)}(\bz')$ (the summation indices are appropriately relabeled), it suffices to decompose the difference into a sum of $R_{<N}$-values and then apply \cref{lem:error}.
\end{proof}
\begin{proof}[Proof of \cref{thm:duality_for_MPL}]
By combining \cref{prop:asymptotic_for_modifiedMPL}, \cref{thm:main,thm:asymp_duality}, we have
\begin{align*}
\Li_{\bk}^{\sh,<N}(\bz)-(-1)^{\iota(\bz)}\Li_{\bk'}^{\sh,<N}(\bz')
&=\widetilde{\Li}_{\bk}^{\sh,(N)}(\bz)-(-1)^{\iota(\bz)}\widetilde{\Li}_{\bk'}^{\sh,(N)}(\bz')+o(1)\\
&=(-1)^{\dep(\bk)}\ii_{\bk}^{(N)}(\bz)-(-1)^{\iota(\bz)+\dep(\bk')}\ii_{\bk'}^{(N)}(\bz')+o(1)\\
&=o(1).
\end{align*}
Therefore, by taking the limit $N\to\infty$, we have the duality.
\end{proof}
% --------------------------------------------------------------------------
\section{Relations for finite multiple zeta values derived from \cref{thm:main}}\label{sec:FMZV}
% --------------------------------------------------------------------------
In \cite{MaesakaSekiWatanabe-pre}, new proofs of both the duality for MZVs and the duality for finite multiple zeta values were provided using \cref{thm:MSW}.
Furthermore, in the previous section, the new proof of the duality for MZVs was extended to a new proof of the duality for MPLs.
Consequently, one might hope that our main result could yield a new proof of the duality for finite multiple polylogarithms (= \cite[Theorem~1.3 (1), Theorem~3.12]{SakugawaSeki2017}).
However, since the left-hand side of \eqref{eq:main} is $\widetilde{\Li}_{\bk}^{\sh,(N)}(\bx)$ rather than $\Li_{\bk}^{\sh,<N}(\bx)$, unfortunately, employing a similar argument to \cite[Section~3]{MaesakaSekiWatanabe-pre} does not yield a relation for finite multiple polylogarithms.
Nevertheless, from our main result, we are able to derive some relations among finite multiple zeta values that we will discuss below.

Let $\zeta_{<N}^{}(\bk)$ denote $\Li_{\bk}^{\sh,<N}(\{1\}^{\dep(\bk)})$.
After Kaneko and Zagier, for an index $\bk$, the \emph{finite multiple zeta value} (FMZV) $\zeta_{\cA}^{}(\bk)$ is defined as
\[
\zeta_{\cA}^{}(\bk)\coloneqq(\zeta_{<p}^{}(\bk)\bmod{p})_{p\in\PP}\in\cA,
\]
where $\PP$ is the set of all prime numbers and
\[
\cA\coloneqq\left.\left(\prod_{p\in\PP} \ZZ/p\ZZ\right) \middle/ \left(\bigoplus_{p\in\PP} \ZZ/p\ZZ\right) \right..
\]
It is known that a certain kind of duality for FMZVs holds as follows.
For two indices $\bk$ and $\bl$, the relation $\bl\preceq \bk$ means that $\bl$ is obtained by replacing some commas in $\bk = (k_1, \dots, k_r)$ by plus signs.
\begin{theorem}[{Hoffman~\cite[Theorem~4.7]{Hoffman2015}}]\label{thm:nonstarduality}
For an index $\bk$, we have
\[
\zeta_{\cA}^{}(\bk)=(-1)^{\dep(\bk)}\sum_{\bk\preceq\bl}\zeta_{\cA}^{}(\bl).
\]
\end{theorem}
Following \cite{MaesakaSekiWatanabe-pre}, when setting $N=p$ in \eqref{eq:main}, all variables disappear, leaving us with merely \cref{thm:nonstarduality}.
However, by setting $N=p-1$ in \eqref{eq:modified_main}, a generalization with variables is obtained.
\begin{theorem}\label{thm:Hoffman_with_variables}
Let $p$ be a prime number, $(k_1,\dots,k_r)$ an index, and $(x_1,\dots,x_r)$ a tuple of indeterminates.
Then we have
\begin{align*}
&\sum_{0<n_1<\cdots<n_r<p}\frac{1}{n_1^{k_1}\cdots n_r^{k_r}}\left[\prod_{i=1}^{r-1}\frac{(x_{i+1}+1)_{n_i}}{(x_i+1)_{n_i}}\right]\frac{(n_r)!}{(x_{r}+1)_{n_r}}\\
&\equiv(-1)^r\sum_{\substack{0< n_{j,1}\leq\cdots\leq n_{j,k_j}<p \ (1\leq j\leq r)\\ n_{j,k_j}<n_{j+1,1} \ (1\leq j<r)}}\prod_{j=1}^r\frac{1}{(n_{j,1}+x_j)n_{j,2}\cdots n_{j,k_j}}\pmod{p},
\end{align*}
where $(x)_n$ denotes the rising factorial, that is, $(x)_n=x(x+1)\cdots(x+n-1)$.
\end{theorem}
In particular, by comparing coefficients in the case of a single variable, the following relations among FMZVs can be obtained.

To state the theorem and for its proof, we introduce some notation.
For a tuple of non-negative integers $\bl=(l_1,\dots,l_r)$ and an index $\bk=(k_1,\dots, k_r)$, we set
\[
\bl\oplus\bk\coloneqq(l_1+k_1,\dots,l_r+k_r),\quad\bl\oslash\bk\coloneqq(l_1+1,\{1\}^{k_1-1},\dots,l_r+1,\{1\}^{k_r-1}),
\]
and $\wt(\bl)\coloneqq l_1+\cdots+l_r$.
The symbol $\bk^{\star}$ denotes the formal sum $\sum_{\bh\preceq\bk}\bh$.
Let $\mathcal{R}$ be the space of formal $\QQ$-linear combinations of indices.
A $\QQ$-bilinear mapping $\underline{\ast}\colon\mathcal{R}\times\mathcal{R}\to\mathcal{R}$ is defined by $\bk \ \underline{\ast} \ \bl=((\bk_{-} \ast \bl),k)+((\bk_{-}\ast\bl_{-}),k+l)$,
where $\bk=(\bk_{-},k)$ and $\bl=(\bl_{-},l)$.
Here, $\ast$ is the usual harmonic product, that is, $\zeta_{<N}^{}(\bk)\zeta_{<N}^{}(\bl)=\zeta_{<N}^{}(\bk\ast\bl)$.
For a positive integer $N$ and an index $\bk$, set $s_{<N}(\bk)\coloneqq\zeta_{<N+1}^{}(\bk)-\zeta_{<N}^{}(\bk)$ and $\zeta_{\leq N}^{\star}(\bk)\coloneqq\zeta_{<N+1}^{}(\bk^{\star})$.
For $N$, $\bk$, and an index $\bl$, we can check that
\[
s_{<N}^{}(\bk)\zeta_{<N+1}^{}(\bl)=s_{<N}^{}(\bk \ \underline{\ast} \ \bl)
\]
holds.
Here, we consider $\zeta_{<N}^{}$, $s_{<N}^{}$, and $\zeta_{\cA}^{}$ as being extended as mappings over $\mathcal{R}$, $\QQ$-linearly.
\begin{theorem}\label{thm:new?relation}
Let $\bk$ be an index and $m$ a non-negative integer.
Then we have
\[
\zeta_{\cA}^{}(\bk \ \underline{\ast} \ (\{1\}^m)^{\star})=(-1)^{\dep(\bk)}\sum_{\substack{\bl\in\ZZ_{\geq 0}^{\dep(\bk)} \\ \wt(\bl)=m}}\sum_{\bl \oplus\bk\preceq\bh\preceq\bl\oslash\bk}\zeta_{\cA}^{}(\bh).
\]
For the case $m=0$, the left-hand side is interpreted as $\zeta_{\cA}^{}(\bk)$.
\end{theorem}
\begin{proof}
Let $p$ be a prime number and $\bk=(k_1,\dots,k_r)$ an index.
In \cref{thm:Hoffman_with_variables}, by setting $x_1=\cdots=x_r=x$, we obtain the following congruence:
\begin{align*}
&\sum_{0<n_1<\cdots<n_r<p}\frac{1}{n_1^{k_1}\cdots n_r^{k_r}}\frac{(n_r)!}{(x+1)_{n_r}}\\
&\equiv(-1)^r\sum_{\substack{0< n_{j,1}\leq\cdots\leq n_{j,k_j}<p \ (1\leq j\leq r)\\ n_{j,k_j}<n_{j+1,1} \ (1\leq j<r)}}\prod_{j=1}^r\frac{1}{(n_{j,1}+x)n_{j,2}\cdots n_{j,k_j}}\pmod{p}.
\end{align*}
By an expansion
\begin{align*}
\frac{(n_r)!}{(x+1)_{n_r}}&=\prod_{i=1}^{n_r}\left(\frac{x}{i}+1\right)^{-1}=\prod_{i=1}^{n_r}\sum_{m_i=0}^{\infty}\frac{(-x)^{m_i}}{i^{m_i}}\\
&=\sum_{m=0}^{\infty}(-x)^m\sum_{m_1+\cdots+m_{n_r}=m}\frac{1}{1^{m_1}2^{m_2}\cdots n_r^{m_{n_r}}}=\sum_{m=0}^{\infty}\zeta_{\leq n_r}^{\star}(\{1\}^m)(-x)^m,
\end{align*}
we compute
\begin{align*}
\sum_{0<n_1<\cdots<n_r<p}\frac{1}{n_1^{k_1}\cdots n_r^{k_r}}\frac{(n_r)!}{(x+1)_{n_r}}&=\sum_{n_r=1}^{p-1}\sum_{m=0}^{\infty} s_{<n_r}(\bk)\zeta_{\leq n_r}^{\star}(\{1\}^m)(-x)^m\\
&=\sum_{m=0}^{\infty}\sum_{n_r=1}^{p-1}s_{<n_r}(\bk \ \underline{\ast} \ (\{1\}^m)^{\star})(-x)^m\\
&=\sum_{m=0}^{\infty}\zeta_{<p}^{}(\bk \ \underline{\ast} \ (\{1\}^m)^{\star})(-x)^m.
\end{align*}
On the other hand, by a simple expansion $(n+x)^{-1}=\sum_{l=0}^{\infty}(-x)^l/n^{l+1}$,
we compute
\begin{align*}
&\sum_{\substack{0< n_{j,1}\leq\cdots\leq n_{j,k_j}<p \ (1\leq j\leq r)\\ n_{j,k_j}<n_{j+1,1} \ (1\leq j<r)}}\prod_{j=1}^r\frac{1}{(n_{j,1}+x)n_{j,2}\cdots n_{j,k_j}}\\
&=\sum_{m=0}^{\infty}(-x)^m\sum_{\substack{l_1+\cdots+l_r=m \\ l_j\geq 0 \ (1\leq j\leq r)}}\sum_{\substack{0< n_{j,1}\leq\cdots\leq n_{j,k_j}<p \ (1\leq j\leq r)\\ n_{j,k_j}<n_{j+1,1} \ (1\leq j<r)}}\prod_{j=1}^r\frac{1}{n_{j,1}^{l_j+1}n_{j,2}\cdots n_{j,k_j}}.
\end{align*}
Since
\[
\sum_{\substack{0< n_{j,1}\leq\cdots\leq n_{j,k_j}<p \ (1\leq j\leq r)\\ n_{j,k_j}<n_{j+1,1} \ (1\leq j<r)}}\prod_{j=1}^r\frac{1}{n_{j,1}^{l_j+1}n_{j,2}\cdots n_{j,k_j}}=\sum_{\bl \oplus\bk\preceq\bh\preceq\bl\oslash\bk}\zeta_{<p}^{}(\bh)
\]
holds, we have the conclusion.
\end{proof}
This theorem might be considered as a finite analogue of the relations among MZVs derived by Kawashima (\cite[Proposition~5.3]{Kawashima-pre}), due to the somewhat similar form of the expressions.
% --------------------------------------------------------------------------
\section{Extended double shuffle relations for multiple polylogarithms}\label{sec:EDSR}
% --------------------------------------------------------------------------
In \cite{Seki-pre}, a quite simple proof of the extended double shuffle relations (EDSR) for MZVs, not utilizing integrals and using \cref{thm:MSW}, is provided by the third author.

His proof can be summarized as follows:
The double shuffle relations (DSR), which is needed for the proof of the EDSR, is usually proved using \eqref{eq:IIEforMZV}.
\cref{thm:MSW} allows for an alternative proof of the DSR based on manipulations of finite sums.
As the sums are finite, this manipulations are possible even for non-admissible indices; this extension of the DSR is referred to as the asymptotic double shuffle relations (ADSR). 
The proof using \eqref{eq:IIEforMZV} is only valid for admissible indices.
Hence, in the typical proof of the EDSR (such as \cite{IharaKanekoZagier2006}), two types of regularization of MZVs are introduced, and the regularization theorem (Reg), which compares them, is proved.
Then the EDSR is proved by combining the DSR and the Reg.
In the new proof, the EDSR can be easily derived from the ADSR, and in this process, neither the shuffle regularization nor the Reg is necessary.

In this section, we extend his proof to offer a simple proof of the extended double shuffle relations for MPLs.
While a description exactly similar to \cref{thm:EDSR_for_MPL} may not be found, essentially the same has been studied by Goncharov~\cite{Goncharov-pre}, Racinet~\cite{Racinet2002}, and Arakawa--Kaneko~\cite{ArakawaKaneko2004}.
% --------------------------------------------------------------------------
\subsection{Notation and the statement}
% --------------------------------------------------------------------------
Let $N$ denote a positive integer.
For each complex number $z$, we prepare an indeterminate $e_z$, and consider the non-commutative polynomial ring $\fH\coloneqq\QQ\langle e_z \mid z\in\CC\rangle$.
Let $e_{z,k}\coloneqq e_ze_0^{k-1}$ for each complex number $z$ and each positive integer $k$.
For each index $\bk=(k_1,\dots, k_r)$ and a tuple of complex numbers $\bz=(z_1,\dots, z_r)$, we put $e_{\bz,\bk}\coloneqq e_{z_1,k_1}\cdots e_{z_r,k_r}$.
We define a subspace $\fH^1$ of $\fH$ as
\[
\fH^1\coloneqq\QQ+\sum_{z\in\CC^{\times}}e_z\fH.
\]
We also define a subring $\fH_{\sh}$ of $\fH$ and subspaces $\fH^1_{\sh}$ and $\fH^0_{\sh}$ as follows:
\begin{align*}
\fH_{\sh}\coloneqq\QQ\langle e_0, e_z \mid |z|\geq 1\rangle
\supset \ &\fH_{\sh}^1\coloneqq\QQ+\sum_{|z|\geq 1}e_z\fH_{\sh}\\
\supset \ &\fH_{\sh}^0\coloneqq\QQ+\sum_{|z|\geq 1}e_z\fH_{\sh}e_0+\sum_{|z|,|w|\geq 1, w\neq 1}e_z\fH_{\sh}e_w.
\end{align*}
We define a $\QQ$-linear mapping $\top\colon\fH^1_{\sh}\to\fH$ by $\top(1)\coloneqq1$ and 
\[
\top(e_{z_1,k_1}\cdots e_{z_r,k_r})\coloneqq e_{z_2/z_1,k_1}\cdots e_{z_r/z_{r-1},k_{r-1}}e_{z_r^{-1},k_r}.
\]
Set $\fH_*^1\coloneqq\top(\fH_{\sh}^1)\subset\fH$.
The \emph{harmonic product} $*$ on $\fH^1$ is defined by rules $w*1=1*w=w$ for any word $w\in\fH^1$, and 
\[
we_{\xi_1,k_1}*w'e_{\xi_2,k_2}=(w*w'e_{\xi_2,k_2})e_{\xi_1,k_1}+(we_{\xi_1,k_1}*w')e_{\xi_2,k_2}+(w*w')e_{\xi_1\xi_2,k_1+k_2}
\]
for any words $w, w'\in\fH^1$, $k_1, k_2\in\ZZ_{>0}$, and $\xi_1, \xi_2\in\CC^{\times}$, with $\QQ$-bilinearity.
The \emph{shuffle product} $\sh$ on $\fH_{\sh}$ is defined by rules $w\sh 1=1\sh w=w$ for any word $w\in\fH_{\sh}$, and
\[
wu_1\sh w'u_2=(w\sh w'u_2)u_1+(wu_1\sh w')u_2
\]
for any words $w, w'\in\fH_{\sh}$, $u_1, u_2\in\{e_z\mid z=0 \text{ or } |z|\geq 1\}$, with $\QQ$-bilinearity.
We can check that the image of $\fH_*^1\times \fH_*^1$ under $*$ is included in $\fH_*^1$ and the image of $\fH_{\sh}^1\times \fH_{\sh}^1$ under $\sh$ is included in $\fH_{\sh}^1$.

For an index $\bk=(k_1,\dots,k_r)$ and a tuple of complex numbers $\bxi=(\xi_1,\dots,\xi_r)$, we define $\Li_{\bk}^{*,<N}(\bxi)$ by
\[
\Li_{\bk}^{*,<N}(\bxi)\coloneqq\sum_{0<n_1<\cdots<n_r<N}\frac{\xi_1^{n_1}\cdots \xi_r^{n_r}}{n_1^{k_1}\cdots n_r^{k_r}}.
\]
If $(k_r,\xi_r)\neq(1,1)$ and $\left|\prod_{j=i}^r\xi_j\right|\leq 1$ for all $1\leq i\leq  r$, then the limit $\displaystyle\lim_{N\to\infty}\Li_{\bk}^{*,<N}(\bxi)$ exists and the limit value is denoted by
\[
\Li_{\bk}^{*}(\bxi)=\sum_{0<n_1<\cdots<n_r}\frac{\xi_1^{n_1}\cdots \xi_r^{n_r}}{n_1^{k_1}\cdots n_r^{k_r}}.
\]
There is a simple relationship between the two types of multiple polylogarithm symbols $\Li_{\bk}^*$ and $\Li_{\bk}^{\sh}$:
\begin{align*}
\Li_{\bk}^*(\xi_1,\dots,\xi_r)&=\Li_{\bk}^{\sh}\left(\prod_{j=1}^r\xi_j^{-1},\prod_{j=2}^r\xi_j^{-1},\dots,\xi_{r-1}^{-1}\xi_r^{-1},\xi_r^{-1}\right),\\
\Li_{\bk}^{\sh}(z_1,\dots,z_r)&=\Li_{\bk}^*\left(\frac{z_2}{z_1},\dots,\frac{z_r}{z_{r-1}},\frac{1}{z_r}\right).
\end{align*}
The notation $\log^{\bullet}N$ means the existence of some positive integer $m$ independent of $N$, represented as $\log^m N$.
\begin{proposition}[cf.~{\cite[Corollaire~2.1.8]{Racinet2002}}]\label{prop:reg_polynomial}
We assume $\left|\prod_{j=i}^r\xi_j\right|\leq 1$ for all $1\leq i\leq  r$.
Then there exists a polynomial $\rL_{\bk,\bxi}^*(x)\in\CC[x]$ such that
\[
\Li_{\bk}^{*,<N}(\bxi)=\rL_{\bk,\bxi}^*(\log N+\gamma)+O(N^{-1}\log^{\bullet}N)
\]
as $N\to\infty$.
Here $\gamma$ is Euler's constant.
Furthermore, the coefficient of $x^i$ in $\rL_{\bk,\bxi}^*(x)$ can be expressed as a $\QQ$-linear combination of converging multiple polylogarithms associated with indices of weight $\wt(\bk)-i$.
In particular, the coefficient of $x^{\wt(\bk)}$ is a rational number.
\end{proposition}
\begin{proof}
The proof is standard, so it is described in a sketchy manner.
First, prove
\[
\Li_{\bk}^{*,<N}(\bxi)=\Li_{\bk}^*(\bxi)+O(N^{-1}\log^{\bullet}N)
\]
as $N\to\infty$ in the case where $(k_r,\xi_r)\neq(1,1)$ is satisfied.
Then, in the case $\bk=(\bk',\{1\}^s)$ and $\bxi=(\bxi',\{1\}^s)$ for some positive integer $s$ and some pair $(\bk',\bxi')$ satisfying the convergence condition, prove the desired assertion by induction on $s$ based on the decomposition of $\zeta_{<N}^{}(1)\Li_{(\bk',\{1\}^{s-1})}^{*,<N}(\bxi',\{1\}^{s-1})$ using the harmonic product formula (\cref{prop:harmonic} below).
\end{proof}
Let $\rL^*(e_{\bxi,\bk})$ be the constant term of $\rL_{\bk,\bxi}^*(x)$ for each $e_{\bxi,\bk}\in\fH_*^1$, and together with $\rL^*(1)\coloneqq1$, extend it to a $\QQ$-linear mapping $\rL^*\colon\fH_*^1\to\CC$.
Any image of an element in $\fH_*^1$ under $\rL^*$ can be expressed as a $\QQ$-linear combination of converging multiple polylogarithms.

The purpose of this section is to provide a simple proof of the following theorem.
\begin{theorem}[Extended double shuffle relations for MPLs]\label{thm:EDSR_for_MPL}
For any $w_1\in\fH_{\sh}^1$, $w_0\in\fH_{\sh}^0$, we have
\[
\rL^*(\top(w_1)*\top(w_0)-\top(w_1\sh w_0))=0.
\]
\end{theorem}
% --------------------------------------------------------------------------
\subsection{Product formulas}
% --------------------------------------------------------------------------
The $\QQ$-linear mapping $\rL_{<N}\colon\fH_*^1\to\CC$ is defined by $\rL_{<N}(1)\coloneqq1$ and $\rL_{<N}(e_{\bxi,\bk})\coloneqq\Li_{\bk}^{*,<N}(\bxi)$.
\begin{proposition}[Harmonic product formula]\label{prop:harmonic}
For $y$, $y'\in\fH_*^1$, we have
\[
\rL_{<N}(y)\rL_{<N}(y')=\rL_{<N}(y*y').
\]
\end{proposition}
\begin{proof}
This is straightforward and well-known.
\end{proof}
The $\QQ$-linear mapping $\ii^{(N)}\colon\fH_{\sh}^1\to\CC$ is defined by $\ii^{(N)}(1)\coloneqq1$ and $\ii^{(N)}(e_{\bz,\bk})\coloneqq\ii_{\bk}^{(N)}(\bz)$.
\begin{proposition}[Asymptotic shuffle product formula]\label{prop:asymp_shuffle}
For $w_1\in\fH_{\sh}^1$, $w_0\in\fH_{\sh}^0$, we have
\[
\ii^{(N)}(w_1)\ii^{(N)}(w_0)=\ii^{(N)}(w_1\sh w_0)+O(N^{-1}\log^{\bullet}N)
\]
as $N\to\infty$.
\end{proposition}
\begin{proof}
It is understood that $\ii^{(N)}$ satisfies the shuffle product formula up to error terms through the same mechanism used to prove the shuffle product formula for MPLs using their iterated integral expressions.
(The difference lies in decomposing the range of summation rather than decomposing the range of integration.)
All error terms can be handled by \cref{lem:error}.
For details, refer to \cite[Propositions~2.3 and 2.4]{Seki-pre} as the procedure is almost the same, if necessary.
Note that terms of the form $(n-Nz)(n-Nz')$ do not appear in each $\ii_{\bk}^{(N)}(\bz)$, and that $w_0$ is an element in $\fH_{\sh}^0$.
\end{proof}
% --------------------------------------------------------------------------
\subsection{Proof of \cref{thm:EDSR_for_MPL}}
% --------------------------------------------------------------------------
\begin{theorem}[Asymptotic double shuffle relations]\label{thm:asymp_DSR}
For $w_1\in\fH_{\sh}^1$, $w_0\in\fH_{\sh}^0$, we have
\[
\rL_{<N}(\top(w_1)*\top(w_0)-\top(w_1\sh w_0))=O(N^{-1/3}\log^{\bullet}N)
\]
as $N\to\infty$.
\end{theorem}
\begin{proof}
By \cref{prop:harmonic}, we see that 
\[
\rL_{<N}(\top(w_1))\rL_{<N}(\top(w_0))=\rL_{<N}(\top(w_1)\ast\top(w_0))
\]
holds.
On the other hand, by \cref{prop:asymp_shuffle}, \cref{thm:main} and \cref{prop:asymptotic_for_modifiedMPL}, we have
\[
\rL_{<N}(\top(w_1))\rL_{<N}(\top(w_0))=\rL_{<N}(\top(w_1\sh w_0))+O(N^{-1/3}\log^{\bullet}N)
\]
as $N\to\infty$.
By combining these two, the conclusion is obtained.
\end{proof}
\begin{proof}[Proof of \cref{thm:EDSR_for_MPL}]
The proof deriving \cref{thm:EDSR_for_MPL} from \cref{thm:asymp_DSR} and \cref{prop:reg_polynomial} is exactly the same as in \cite[Section~3]{Seki-pre}.
\end{proof}
% --------------------------------------------------------------------------
\section{Miscellaneous}\label{sec:Misc}
% --------------------------------------------------------------------------
In this section, we provide a proof of \eqref{eq:pi/4} mentioned in \cref{sec:intro} and a reformulation of special cases of our main result using certain types of multiple harmonic sums that do not involve binomial coefficients.
This is a generalization of the equation \eqref{eq:MSWexample}, and the authors initially discovered these formulas through numerical experiments.
\begin{proof}[Proof of \eqref{eq:pi/4}]
In this proof, $i$ denotes the imaginary unit.
By employing \cref{thm:main}, for $x = \pm i$, we obtain
\[
\sum_{n=1}^{N-1} \frac{1}{n} \frac{\binom{N-1}{n}}{\binom{\pm Ni - 1}{n}} = - \sum_{n=1}^{N-1} \frac{1}{n \mp Ni}.
\]
Combining these yields
\begin{align*}
\sum_{n=1}^{N-1} \frac{N}{n^2+N^2} &= \frac{1}{2i} \sum_{n=1}^{N-1} \frac{1}{n} \left(\frac{\binom{N-1}{n}}{\binom{-Ni-1}{n}} - \frac{\binom{N-1}{n}}{\binom{Ni-1}{n}} \right)\\
&= \frac{1}{2i} \sum_{n=1}^{N-1} \frac{1}{n} \left(\prod_{j=1}^n \frac{N-j}{N^2+j^2}\right) \left(\prod_{l=1}^n (Ni-l) - \prod_{l=1}^n (-Ni-l) \right).
\end{align*}
By the definition of the Stirling number of the first kind, we have
\[
\prod_{l=1}^n (\pm Ni-l) = (-1)^n \sum_{l=0}^n \stirlingI{n+1}{l+1} (\mp Ni)^l,
\]
which implies
\[
\frac{1}{2i}\left(\prod_{l=1}^n (Ni-l) - \prod_{l=1}^n (-Ni-l)\right)=(-1)^n\sum_{0 \le l < n/2} (-1)^{l+1} \stirlingI{n+1}{2l+2} N^{2l+1}.
\]
Thus, the proof completes.
\end{proof}
When $r=1$ and $x_1=-1$, the equation \eqref{eq:modified_main} becomes
\begin{equation}\label{eq:depone_case}
\sum_{n=1}^N\frac{(-1)^n}{n^k}\frac{\binom{N}{n}}{\binom{N+n}{n} } =
-\sum_{1\leq n_1\leq \cdots \leq n_k\leq N}
\frac{1}{(n_1+N)n_2\cdots n_k}.
\end{equation}
We give another expression for this quantity.
\begin{theorem}
Let $N$ and $k$ be positive integers.
Then we have
\[
\sum_{1\leq n_1\leq \cdots \leq n_{k}\leq N}
\frac{1}{(n_1+N)n_2\cdots n_{k}}
=
\begin{cases}
\displaystyle{\frac{1}{2}\sum_{1\leq m_1\leq\cdots\leq m_r\leq N}
\frac{1}{m_1^2\cdots m_r^2}} &\text{if } k=2r,\\
\displaystyle{\sum_{1\leq n\leq 2m_1 \leq \cdots\leq 2m_r\leq 2N}
\frac{(-1)^{n-1}}{nm_1^2\cdots m_r^2}} &\text{if } k=2r+1.
\end{cases}
\]
\end{theorem}
\begin{proof}
Let $a_{N,k}$ and $b_{N,k}$ be the left and right hand side, respectively.
We prove the claim by induction on $N$ and $k$.
The case $N=1$ follows from $a_{1,k}=\frac{1}{2}=b_{1,k}$. The case $k=1$,
\begin{equation}\label{eq:MSWexample_second}
\sum_{n=1}^{N}\frac{1}{n+N}=\sum_{n=1}^{2N}\frac{(-1)^{n-1}}{n},
\end{equation}
is equivalent to \eqref{eq:MSWexample}.
Let $N>1$ and $k>1$. Then we have
\begin{align*}
a_{N,k} & =\sum_{1\leq n_1\leq\cdots\leq n_k\leq N-1}\frac{1}{(n_1+N)n_2\cdots n_k}
+\sum_{1\leq n_1\leq\cdots\leq n_k=N}\frac{1}{(n_1+N)n_2\cdots n_k}\\
 & =\sum_{1\leq n_1\leq\cdots\leq n_k\leq N-1}\frac{1}{(n_1+N)n_2\cdots n_k}+\frac{1}{N}\sum_{1\leq n_1\leq\cdots\leq n_{k-1}\leq N}\frac{1}{(n_1+N)n_2\cdots n_{k-1}}.
\end{align*}
Since
\[
\sum_{n_1=1}^{n_2}\frac{1}{n_1+N}=\sum_{n_1=1}^{n_2}\frac{1}{n_1+N-1}+\frac{1}{n_2+N}-\frac{1}{N}=
\sum_{n_1=1}^{n_2}\frac{1}{n_1+N-1}-\frac{n_2}{N(n_2+N)},
\]
we have
\begin{align*}
&\sum_{1\leq n_1\leq\cdots\leq n_k\leq N-1}\frac{1}{(n_1+N)n_2\cdots n_k}\\
& =\sum_{1\leq n_1\leq\cdots\leq n_k\leq N-1}\frac{1}{(n_1+N-1)n_2\cdots n_k}-\frac{1}{N}\sum_{1\leq n_2\leq\cdots\leq n_k\leq N-1}\frac{1}{(n_2+N)n_3\cdots n_k}\\
 & =a_{N-1,k}-\frac{1}{N}\sum_{1\leq n_1\leq\cdots\leq n_{k-1}\leq N-1}\frac{1}{(n_1+N)n_2\cdots n_{k-1}}.
\end{align*}
Thus, by induction hypothesis, we have
\begin{align*}
a_{N,k} & =a_{N-1,k}+\frac{1}{N}\sum_{1\leq n_1\leq\cdots\leq n_{k-1}=N}\frac{1}{(n_1+N)n_2\cdots n_{k-1}}\\
&=a_{N-1,k}+\frac{1}{N^2}a_{N,k-2}\\
&=b_{N-1,k}+\frac{1}{N^2}b_{N,k-2}
=b_{N,k}.
\end{align*}
Here, we have set $a_{N,0}=b_{N,0}=\frac{1}{2}$.
This completes the proof.
\end{proof}
\begin{remark}
We can also directly prove that the left-hand side of \eqref{eq:depone_case} equals $-b_{N,k}$ using the method of connected sums.
We consider a connected sum
\[
\sum_{1\leq n\leq m_1\leq \cdots\leq m_b\leq N}\frac{(-1)^n}{n^a}\frac{\binom{m_1}{n}}{\binom{m_1+n}{n}}\frac{1}{m_1^2\cdots m_b^2}
\]
for positive integers $a$ and $b$.
After repeatedly applying the transport relation \cite[Lemma~2.1 (2.4)]{Hessami-etal2014},
\[
\frac{1}{n^2}\frac{\binom{m}{n}}{\binom{m+n}{n}}=\sum_{n\leq m'\leq m}\frac{\binom{m'}{n}}{\binom{m'+n}{n}}\frac{1}{(m')^2},
\]
the desired formula is proved by applying
\[
\sum_{1\leq n\leq m}(-1)^n\frac{\binom{m}{n}}{\binom{m+n}{n}}=-\frac{1}{2}
\]
(this follows from \cref{lem:Sum-Binom}) once when $k$ is even and applying 
\[
\sum_{1\leq n\leq m}\frac{(-1)^n}{n}\frac{\binom{m}{n}}{\binom{m+n}{n}}=\sum_{1\leq n\leq 2m}\frac{(-1)^n}{n}
\]
(this follows from \eqref{eq:depone_case} for the case $k=1$ and \eqref{eq:MSWexample_second}) once when $k$ is odd.
When $k$ is even, this was already proved in \cite[Corollary 2.4]{Hessami-etal2014}.
\end{remark}
% --------------------------------------------------------------------------


\begin{thebibliography}{BBBL}
\bibitem[AK]{ArakawaKaneko2004}
T. Arakawa, M. Kaneko, \emph{On multiple $L$-values}, J. Math.~Soc.~Japan \textbf{56} (2004), 967--991.
\bibitem[BBBL]{BorweinEtal2001}
J. M. Borwein, D. M. Bradley, D. J. Broadhurst, P. Lison\v{e}k, \emph{Special values of multiple polylogarithms}, Trans.~Amer.~Math.~Soc.~\textbf{353} (2001), 907--941.
\bibitem[G1]{Goncharov1995}
A. B. Goncharov, \emph{Polylogarithms in arithmetic and geometry}, Proc.~ICM (Z\"urich, 1994), Birkh\'auser Verlag, Basel, 1995, 374--387.
\bibitem[G2]{Goncharov-pre}
A. B. Goncharov, \emph{Multiple polylogarithms and mixed Tate motives}, preprint, arXiv:math/\-0103059.
\bibitem[HHT]{Hessami-etal2014}
Kh. Hessami Pilehrood, T. Hessami Pilehrood, R. Tauraso, \emph{New properties of multiple harmonic sums modulo $p$ and $p$-analogues of Leshchiner's series}, Trans.~Amer.~Math.~Soc.~\textbf{366} (2014), 3131--3159.
\bibitem[HMO]{HiroseMuraharaOnozuka-pre}
M. Hirose, H. Murahara, T. Onozuka, \emph{Integral expressions for Schur multiple zeta values}, preprint, arXiv:2209.04858.
\bibitem[H]{Hoffman2015}
M. E. Hoffman, \emph{Quasi-symmetric functions and mod $p$ multiple harmonic sums}, Kyushu J. Math.~\textbf{69} (2015), 345--366.
\bibitem[IKZ]{IharaKanekoZagier2006}
K. Ihara, M. Kaneko, D. Zagier, \emph{Derivation and double shuffle relations for multiple zeta values}, Compos.~Math.~\textbf{142} (2006), 307--338.
\bibitem[KMS]{KawamuraMaesakaSeki2022}
H. Kawamura, T. Maesaka, S. Seki, \emph{Multivariable connected sums and multiple polylogarithms}, Res.~Math.~Sci.~\textbf{9} (2022), 4, 25 pp.
\bibitem[K]{Kawashima-pre}
G. Kawashima, \emph{Multiple series expressions for the Newton series which interpolate finite multiple harmonic sums}, preprint, arXiv:0905.0243.
\bibitem[LZ]{LinebargerZhao2015}
E. Linebarger, J. Zhao, \emph{A family of multiple harmonic sum and multiple zeta star value
identities}, Mathematika \textbf{61} (2015), 63--71.
\bibitem[MSW]{MaesakaSekiWatanabe-pre}
T. Maesaka, S. Seki, T. Watanabe, \emph{Deriving two dualities simultaneously from a family of identities for multiple harmonic sums}, preprint, arXiv:2402.05730.
\bibitem[R]{Racinet2002}
G. Racinet, \emph{Doubles m\'elanges des polylogarithmes multiples aux racines de l'unit\'e}, Publ. Math. Inst.~Hautes \'Etudes Sci. (2002), 185--231.
\bibitem[SS]{SakugawaSeki2017}
K. Sakugawa, S. Seki, \emph{On functional equations of finite multiple polylogarithms}, J. Algebra \textbf{469} (2017), 323--357.
\bibitem[S1]{Seki2020}
S. Seki, \emph{Connectors}, RIMS K\^oky\^uroku \textbf{2160} (2020), 15--27.
\bibitem[S2]{Seki-pre}
S. Seki, \emph{A proof of the extended double shuffle relation without using integrals}, preprint, arXiv:2402.18300.
\bibitem[Y1]{Yamamoto2017}
S. Yamamoto, \emph{Multiple zeta-star values and multiple integrals}, RIMS K\^oky\^uroku Bessatsu \textbf{B68} (2017), 3--14.
\bibitem[Y2]{Yamamoto-pre}
S. Yamamoto, \emph{Some remarks on Maesaka--Seki--Watanabe's formula for the multiple harmonic sums}, preprint, arXiv:2403.03498.
\end{thebibliography}
\end{document}